\documentclass[10pt]{amsart}

\usepackage{mathptmx,amsmath, amssymb, amscd, paralist,tabularx,supertabular,amsmath, verbatim, amsthm, amssymb, mathrsfs, manfnt,latexsym,amscd,graphicx,hyperref,color}

\evensidemargin \oddsidemargin
\author{Benjamin Linowitz}
\address{Department of Mathematics\\Oberlin College\\Oberlin, OH 44074}
\email{benjamin.linowitz@oberlin.edu}

\author{D. B. McReynolds}
\address{Department of Mathematics\\Purdue University\\West Lafayette, IN 47907}
\email{dmcreyno@purdue.edu}

\author{Paul Pollack}
\address{Department of Mathematics\\University of Georgia\\Athens, GA 30602}
\email{pollack@uga.edu}

\author{Lola Thompson}
\address{Department of Mathematics\\Oberlin College\\Oberlin, OH 44074}
\email{lola.thompson@oberlin.edu}

\title{Counting and effective rigidity in algebra and geometry}
\usepackage{amsmath,amssymb,amsthm,url}

\usepackage{fancyhdr}
\fancypagestyle{plain}

\DeclareMathAlphabet{\curly}{U}{rsfs}{m}{n}

\DeclareMathOperator{\Br}{Br}

\DeclareMathOperator{\covol}{covol}

\DeclareMathOperator{\disc}{disc}
\DeclareMathOperator{\diag}{diag}

\DeclareMathOperator{\Ram}{Ram}

\DeclareMathOperator{\Frob}{Frob}
\DeclareMathOperator{\Gal}{Gal}

\DeclareMathOperator{\lcm}{lcm}

\DeclareMathOperator{\N}{N}

\DeclareMathOperator{\Norm}{Norm}
\DeclareMathOperator{\nr}{nr}

\DeclareMathOperator{\PSL}{PSL}

\DeclareMathOperator{\Res}{Res}
\DeclareMathOperator{\Reg}{Reg}

\DeclareMathOperator{\SL}{SL}
\DeclareMathOperator{\SO}{SO}
\DeclareMathOperator{\SU}{SU}
\DeclareMathOperator{\MM}{M}

\DeclareMathOperator{\Tr}{Tr}
\DeclareMathOperator{\tr}{tr}

\DeclareMathOperator{\coarea}{coarea}

\newtheorem{thm}{Theorem}[section]
\newtheorem{cor}[thm]{Corollary}

\newtheorem{prop}[thm]{Proposition}
\newtheorem{lem}[thm]{Lemma}
\theoremstyle{definition}

\newtheorem*{rmk}{Remark}

\theoremstyle{remark}
\newtheorem{ex0}{Example}
\newtheorem{ex}{Examples}

\def\1{\mathbf{1}}
\def\disc{\mathrm{disc}}
\def\Prob{\mathbf{Prob}}

\theoremstyle{remark}

\theoremstyle{plain}

\def\farey{\mathfrak{F}}

\def\C{\mathbf{C}}
\def\Q{\mathbf{Q}}

\def\pp{\mathfrak{p}}
\def\qq{\mathfrak{q}}
\def\ff{\mathfrak{f}}
\def\Qq{\curly{Q}}
\def\Pp{\curly{P}}
\def\Rr{\curly{R}}
\def\R{\mathbf{R}}

\def\m{\mathfrak{m}}
\def\N{\mathbf{N}}
\def\Q{\mathbf{Q}}
\def\Z{\mathbf{Z}}
\def\Ss{\curly{S}}
\def\1{\mathbf{1}}
\def\disc{\mathrm{disc}}
\def\Gal{\mathrm{Gal}}
\def\Frob{\mathrm{Frob}}
\def\lcm{\mathop{\operatorname{lcm}}}
\newcommand{\su}{\subset}

\newcommand{\frakp}{\mathfrak{p}}
\newcommand{\frakq}{\mathfrak{q}}

\newcommand{\scrC}{\mathscr{C}}

\newcommand{\calD}{\mathcal{D}}
\newcommand{\calE}{\mathcal{E}}

\newcommand{\calO}{\mathcal{O}}

\newcommand{\abs}[1]{\left\vert#1\right\vert}
\newcommand{\set}[1]{\left\{#1\right\}}
\newcommand{\brac}[1]{\left[#1\right]}
\newcommand{\pr}[1]{\left( #1 \right) }

\newcommand{\iny}{\infty}

\newcommand{\lra} {\longrightarrow}

\def\v{{\bf v}}

\makeatletter
\def\moverlay{\mathpalette\mov@rlay}
\def\mov@rlay#1#2{\leavevmode\vtop{%
   \baselineskip\z@skip \lineskiplimit-\maxdimen
   \ialign{\hfil$\m@th#1##$\hfil\cr#2\crcr}}}
\newcommand{\charfusion}[3][\mathord]{
    #1{\ifx#1\mathop\vphantom{#2}\fi
        \mathpalette\mov@rlay{#2\cr#3}
      }
    \ifx#1\mathop\expandafter\displaylimits\fi}
\makeatother

\makeatletter
\let\@@pmod\pmod
\DeclareRobustCommand{\pmod}{\@ifstar\@pmods\@@pmod}
\def\@pmods#1{\mkern4mu({\operator@font mod}\mkern 6mu#1)}
\makeatother

\begin{document}

\begin{abstract}
\noindent The purpose of this article is to produce effective versions of some rigidity results in algebra and geometry. On the geometric side, we focus on the spectrum of primitive geodesic lengths (resp., complex lengths) for arithmetic hyperbolic 2--manifolds (resp., 3--manifolds). By work of Reid, this spectrum determines the commensurability class of the 2--manifold (resp., 3--manifold). We establish effective versions of these rigidity results by ensuring that, for two incommensurable arithmetic manifolds of bounded volume, the length sets (resp., the complex length sets) must disagree for a length that can be explicitly bounded as a function of volume. We also prove an effective version of a similar rigidity result established by the second author with Reid on a surface analog of the length spectrum for hyperbolic 3--manifolds. These effective results have corresponding algebraic analogs involving maximal subfields and quaternion subalgebras of quaternion algebras. To prove these effective rigidity results, we establish results on the asymptotic behavior of certain algebraic and geometric counting functions which are of independent interest.
\end{abstract}

\maketitle

\section{Introduction}


\subsection{Inverse problems}

%
%

\subsubsection{Algebraic problems}

Given a degree $d$ central division algebra $D$ over a field $k$, the set of isomorphism classes of maximal subfields $\textrm{MF}(D)$ of $D$ is a basic and well studied invariant of $D$.

\textbf{Question 1.} \textsl{Do there exist non-isomorphic, central division algebras $D_1,D_2/k$ with $\mathrm{MF}(D_1) = \mathrm{MF}(D_2)$?}

Restricting to the class of number fields $k$, by a well-known consequence of class field theory, when $D/k$ is a quaternion algebra, $\textrm{MF}(D) = \textrm{MF}(D')$ if and only if $D \cong D'$ as $k$--algebras.  Unfortunately, for most higher degree division algebras, we have $\textrm{MF}(D) = \textrm{MF}(D^{\mathrm{op}})$ and $D \ncong D^{\mathrm{op}}$ where $D^{\mathrm{op}}$ is the opposite algebra for $D$. For a fixed algebra $D/k$, the number of isomorphism classes of algebras $D'/k$ with $\textrm{MF}(D) = \textrm{MF}(D')$ is the \textbf{genus} of $D$ and is finite in this setting; see \cite{CRR}, \cite{GS}, \cite{Meyer} for some recent work on genus of $D/k$ for general fields $k$.



The Brauer group $\Br(k)$ of a field $k$ is the set of Morita equivalence classes $[A]$ of central, simple $k$--algebras. The group operation is given by tensor product $\otimes_k$ with inverses given by $[A^{\mathrm{op}}]$. Each class $[A]$ has a unique central division algebra $D_A$. Given a finite extension $L/k$, we have a homomorphism $\Res_{L/k}\colon \Br(k) \to \Br(L)$ defined by $\Res_{L/k}([A_0]) = [A_0 \otimes_k L]$. For a class $[A] \in \Br(L)$, the set $(\Res_{L/k})^{-1}([A])$ is the set of Morita equivalence classes $[A_0]$ in $\Br(k)$ such that $[A] = [A_0 \otimes_k L]$. Fix a finite extension $L_1/k$ and algebra $A/L_1$.


\textbf{Question 2.} \textsl{Does there exist a finite extension $L_2/k$ and an algebra $A'/L_2$ with $(\Res_{L_1/k})^{-1}(A) = (\Res_{L_2/k})^{-1}(A')$?}

In this generality, we cannot hope to conclude that $L_1 \cong L_2$ and $A \cong A'$ (see \cite{McReynolds}). However, when $L_1,L_2/k$ are quadratic extensions and $A$ is a quaternion algebra, we have $L_1 \cong L_2$ provided $(\Res_{L_1/k})^{-1}(A)$ is non-empty.

\subsubsection{Geometric problems}

Given a closed, negatively curved, Riemannian manifold $M$, we have an analytic invariant given by the \textbf{eigenvalue spectrum} $\mathcal{E}(M)$ of the Laplace--Beltrami operator acting on $L^2(M)$. Similarly, we have a geometric invariant given by the \textbf{primitive geodesic length spectrum} $\mathcal{L}_p(M)$ of lengths $\ell$ of primitive, closed geodesics. Both geometric invariants of $M$ are multi-sets of the form
\[ \mathcal{E}(M), \mathcal{L}_p(M) = \set{(\lambda_j,m_{\lambda_j})}, \set{(\ell_j,m_{\ell_j})} \subset \R^{\geq 0} \times \N. \]
The integers $m_{\lambda_j},m_{\ell_j}$ are called the \textbf{multiplicities} and give the dimension of the associated $\lambda_j$--eigenspace or the number of distinct occurrences of the primitive length $\ell_j$, respectively. Forgetting the multiplicities, the set of primitive lengths will be called the \textbf{length set}. Two consequences of the manifold being negatively curved are that every free homotopy class of closed loops contains a unique geodesic representative and the length spectrum is discrete; the eigenvalue spectrum is always discrete when $M$ is closed. Discreteness here means that the sets without multiplicity $\set{\lambda_j},\set{\ell_j}$ are discrete and the multiplicities $m_{\lambda_j},m_{\ell_j}$ are finite. These spectra are closely related (see \cite{Gangolli}). When $M$ is hyperbolic $2$--manifold, these spectra determine one another by Selberg's trace formula (see \cite[Ch 9, \S 5]{Buser}). When $M$ is a hyperbolic $3$--manifold, each closed geodesic can be assigned a complex length $\ell = \ell_0 + \theta i$ where $\theta$ is the angle of rotation and $\ell_0$ is the length of the geodesic. In this case, we have the \textbf{complex length spectrum} $\mathcal{L}_c(M)$ and associated \textbf{complex length set}, and the spectra $\mathcal{E}(M)$, $\mathcal{L}_c(M)$ determine one another by Selberg's trace formula. 


\textbf{Question 3.} \textsl{Do there exist non-isometric Riemannian manifolds $M_1,M_2$ such that $\mathcal{L}_p(M_1) = \mathcal{L}_p(M_2)$ (resp., $\mathcal{L}_c(M_1) = \mathcal{L}_c(M_2)$ or $\mathcal{E}(M_1) = \mathcal{E}(M_2)$)?}

Restricting to the class of closed hyperbolic $n$--manifolds, starting with the constructions of Vign{\'e}ras \cite{Vigneras80A}, Sunada \cite{Sunada}, many papers have constructed arbitrarily large finite families pairwise non-isometric hyperbolic $n$--manifolds with identical eigenvalue and (complex) length spectra for all $n\geq 2$; such manifolds are said to be \textbf{isospectral} or (\textbf{complex}) \textbf{length isospectral}. By construction, the pairs produced by Vign{\'e}ras, Sunada are commensurable; all presently known pairs are commensurable. It is not known if such pairs must be commensurable in any dimension $n \geq 2$. For arithmetic hyperbolic $2$-- or $3$--manifolds, they must be. Reid \cite{Reid-Isospectral} proved that hyperbolic 2--manifolds with identical length spectra are commensurable provided one of the manifolds is arithmetic. Using Selberg's trace formula, Reid also obtained an identical result for eigenvalue spectra. As arithmeticity is a commensurability invariant, the other manifold is also arithmetic. Reid also established this result for arithmetic hyperbolic $3$--manifolds with the length spectrum replaced with the complex length spectrum. Chinburg--Hamilton--Long--Reid \cite{chinburg-geodesics} extended Reid's result on length spectra for arithmetic hyperbolic 2--manifolds to arithmetic hyperbolic 3--manifolds. Prasad--Rapinchuk \cite{PrasadRap} also extended \cite{Reid-Isospectral} for many classes of arithmetic, locally symmetric manifolds. Before \cite{PrasadRap}, it was known that \cite{Reid-Isospectral} could not be extended to general locally symmetric manifolds as Lubotzky--Samuels--Vishne \cite{LSV05} constructed arbitrarily large finite families of pairwise incommensurable, isospectral arithmetic, compact, locally symmetric manifolds with real rank $n$ for all $n$ using a method similar to Vign{\'e}ras. For a fixed manifold, the number of commensurability classes is always finite by \cite{PrasadRap} but can be can be arbitrarily large by \cite{LSV05}. This finiteness is the geometric analog of the finiteness of genera for division algebras over number fields where genera are also arbitrarily large.

\subsection{Main results: Effective rigidity}

We refer the reader to the notation list found at the beginning of \S \ref{Background} for any undefined symbols or terms.

\subsubsection{Geodesics}

Our first result is an effective version of Reid \cite{Reid-Isospectral}.


\begin{thm}\label{thm:effectiveCHLR}
Let $M_1,M_2$ be compact arithmetic hyperbolic 2--manifolds (resp., $3$--manifolds) with volume less than $V$. There exist absolute effectively computable constants $c_1,c_2$ (resp., $c_3$) such that if the length sets (resp., complex length) of $M_1$ and $M_2$ agree for all lengths less than $c_1e^{c_2\log(V)V^{130}}$ (resp., $c_3e^{\left(\log(V)^{\log(V)}\right)}$), then $M_1$ and $M_2$ are commensurable.
\end{thm}

Millichap \cite{Millichap,Millichap1} constructed roughly $(2n)!$ incommensurable hyperbolic 3--manifolds that have the same first $2n+1$ (complex) geodesic lengths. Moreover, the manifolds all have the same volume and the volume of these manifolds grows linearly in $n$. His examples are non-arithmetic and his methods are geometric/topological. Neither of these constructions produce lower bounds near the upper bound we provide in Theorem \ref{thm:effectiveCHLR}. Since the completion of this paper, Futer--Millichap \cite{FM17} and \cite{LMPT17} have produced additional examples of non-arithmetic and arithmetic hyperbolic 2-- and 3--manifolds that share the same geodesic lengths for the first $n$ lengths or any finite subset of lengths. Additionally, both constructions give control on the volumes of the examples.

We now sketch the proof of Theorem \ref{thm:effectiveCHLR} in the case that $M_1,M_2$ are hyperbolic $3$--manifolds with $\pi_1(M_j) = \Gamma_j$ and volumes bounded above by $V$. This proof appears in Section \ref{mainproofsubsection} and makes use of a wide variety of effective results in algebraic number theory. For simplicity, we assume that $M_1,M_2$ are derived from quaternion algebras. To prove Theorem \ref{thm:effectiveCHLR}, it suffices to show that the invariant trace fields and invariant quaternion algebras of $M_1,M_2$ are isomorphic (see for instance \cite[Ch 8.4]{MR}). That the trace fields are isomorphic is relatively straightforward. We show (in Proposition \ref{proposition:derivednonrealeigenvalue}) that there exist hyperbolic $\gamma_1\in\Gamma_1, \gamma_2\in\Gamma_2$, all of whose powers have non-real eigenvalues and whose associated geodesics have the same length. It follows that $\tr(\gamma_1) = \pm \tr(\gamma_2)$ and hence $\gamma_1,\gamma_2$ generate the same extension $k$ of $\Q$.  By \cite[Lemma 2.3]{chinburg-geodesics} this extension $k$ is isomorphic to the trace fields of both $M_1,M_2$.

That $M_1,M_2$ have isomorphic quaternion algebras $B_1,B_2$ is more nuanced. To prove that $B_1,B_2$ are isomorphic it suffices to show that $\mathrm{MF}(B) = \mathrm{MF}(B')$. We prove an effective version of this result in Theorem \ref{theorem:recognizingalgebras} by proving that if two quaternion algebras $B_1,B_2/k$ with $\abs{\disc(B_1)},\abs{\disc(B_2)} < x$ admit embeddings of precisely the same quadratic extensions of $L/k$ with $\abs{\Delta_{L/k}}$ less than some explicit function $f(x)$ (which may involve constants depending on $k$), then $B_1,B_2$ are isomorphic. In order to use Theorem \ref{theorem:recognizingalgebras} with quaternion algebras $B_1,B_2$ for $M_1,M_2$, we first show (in Lemma \ref{lem:Bdiscbound}) that $\abs{\disc(B_1)},\abs{\disc(B_2)}$ are bounded above by an explicit function $g(V)$. Setting $x=g(V)$, let $L/k$ be a quadratic extension which embeds into $B_1$ and has $\abs{\Delta_{L/k}}$ less than $f(g(V))$. Proposition \ref{proposition:finalbounds} shows that there exists $u_1 \in B_1$ with $L=k(u_1)$ that has image $\gamma_1 \in \Gamma_1$ that is associated to a geodesic in $M_1$ of length less than $c_3e^{\left(\log(V)^{\log(V)}\right)}$ for some constant $c$. By hypothesis there exists $\gamma_2 \in \Gamma_2$ associated to a geodesic with the same complex length as $\gamma_1$. The preimage $u_2 \in B_2$ of $\gamma_2$ generates a quadratic extension $k(u_2)/k$ which is isomorphic to $L$. Interchanging the roles of $M_1,M_2$ and applying Theorem \ref{theorem:recognizingalgebras}, we conclude that $B_1,B_2$ are isomorphic.

\subsubsection{Totally geodesic surfaces}

For a finite volume hyperbolic 3--manifold $M$, $GS(M)$ will denote the isometry classes of finite volume, properly immersed, totally geodesic surfaces up to free homotopy. A special case of Thurston \cite[Cor 8.8.6.]{Thurston} implies that $GS(M)$ contains only finitely many Riemann surfaces of a fixed finite topological type; this also follows from a compactness argument. Reid and the second author \cite[Thm 1.1]{McReid} prove that if $M_1,M_2$ are arithmetic hyperbolic $3$--manifolds with $GS(M_1) = GS(M_2) \ne \emptyset$, then $M_1,M_2$ are commensurable. Our second result is an effective version.

\begin{thm}\label{EffectiveMcReid}
Let $M_1,M_2$ be arithmetic hyperbolic $3$--manifolds with volumes less than $V$ and $GS(M_1) \cap GS(M_2) \ne \emptyset$. Then there exists an absolute, effectively computable constant $c$ such that if $GS(M_1),GS(M_2)$ agree for all totally geodesic surfaces with area less than $e^{cV}$, then $M_1$ and $M_2$ are commensurable.
\end{thm}

There are infinitely many commensurability classes $\mathcal{C}$ of arithmetic hyperbolic 3--manifolds such that all $M \in \mathcal{C}$ have $GS(M) = \emptyset$ (see \cite[Cor 7]{MR-GeodesicSurfaces} and \cite[p.~546]{CR-SimpleGeodesics}). However, once an arithmetic hyperbolic 3--manifold has one such surface, it is a well-known fact that there are necessarily infinitely many distinct commensurability classes of such surfaces. Below, Theorem \ref{thm:infinitelymanytotallygeodesics} provides a lower bound for the number of commensurability classes of surfaces up to some volume in these arithmetic hyperbolic $3$--manifold and hence implies the infinitude of such surfaces for those manifolds. To the best of our knowledge the lower bound we provide is the first such lower bound.

\subsubsection{Algebraic}

We now turn to algebraic effective rigidity results which, aside from being independently interesting, provide us with tools for proving the above geometric effective rigidity results. Our first result is an effective version of the fact that quaternion algebras over number fields are determined by their maximal subfields.

\begin{thm}\label{theorem:recognizingalgebras}
Let $k$ be a number field and let $B$, $B^\prime/k$ be quaternion algebras satisfying $\abs{\disc(B)}, \abs{\disc(B^\prime)} < x$. If every quadratic field extension $L/k$ with
\[ \abs{\Delta_{L/k}} < 64^{n_k^3}d_k^{n_k} e^{2n_k\brac{\frac{21x}{\log^3(x)}+x}} \]
embeds into $B$ if and only if it embeds into $B^\prime$, then $B\cong B^\prime$.
\end{thm}

We further note that in Theorem \ref{theorem:recognizingalgebras}, if the quaternion algebras $B,B^\prime$ are both unramified at a common real place of $k$, then we need only consider quadratic extensions $L/k$ which are not totally complex. 	

Our second result is the algebraic counterpart of our effective result, Theorem \ref{EffectiveMcReid}, on totally geodesic surfaces.


\begin{thm}\label{thm:recognizingalgebrasfromsubalgebras}
Suppose that $k$ is a number field and $B_0/k$ is a quaternion algebra. Let $L_1, L_2/k$ be quadratic extensions and define $B_1=B_0\otimes_k L_1$ and $B_2=B_0\otimes_k L_2$. If  $B\otimes_k L_1\cong B_1$ if and only if $B\otimes_k L_2\cong B_2$ for every quaternion algebra $B/k$ with $\Ram_\infty(B)=\Ram_\infty(B_0)$ and satisfying
\[ \abs{\disc(B)}\leq d_\ell^{2C} (2\log\pr{\abs{\disc(B_1)}\abs{\disc(B_2)}})^4\abs{\disc(B_1)}\abs{\disc(B_2)}, \]
then $L_1\cong L_2$ and $B_1\cong B_2$.
\end{thm}

Here $C$ is the (absolute) constant appearing in the bound on the least prime ideal in the Chebotarev density theorem \cite{effectiveCDT}. Theorem \ref{thm:recognizingalgebrasfromsubalgebras} is stronger than the algebraic result deduced in \cite{McReid}.  Consequently, Theorem \ref{thm:recognizingalgebrasfromsubalgebras} provides similar geometric spectral rigidity results but for a broader class of manifolds modeled on $(\mathbf{H}^2)^a \times (\mathbf{H}^3)^b$.


\subsection{Main tools: Counting function and asymptotic behavior}


Understanding the asymptotic behavior of counting functions is central to the field of analytic number theory.  Our work falls within the subfield of arithmetic statistics, which centers around counting problems on number fields and elliptic curves with bounded discriminant. The analytic method used in our proofs goes at least back to Harvey Cohn \cite{cohn54}, who used a similar approach to count the number of abelian cubic extensions of $\Q$ with bounded discriminant. Other seminal works in this area include the classical theorem of Davenport--Heilbronn \cite{davenportheilbronn}, which provides an asymptotic formula for the number of cubic number fields with bounded discriminant, and its various generalizations to certain classes of number fields of higher degree (for an excellent survey of these results we refer the reader to Bhargava's lecture from the 2006 ICM \cite{bhargavaICM}). The general philosophy used in all of this work is to introduce a generating function whose coefficients count the object being studied and then apply a Tauberian theorem to convert information about the analytic behavior of these functions near their singularities into useful information about the counts. Although the results in this subsection are functioning as tools for proving our above stated results, they fall naturally within this larger program of study and thus are of independent interest. Specifically, the technical results that we discuss below all involve counting problems on central division algebras with bounded discriminant.

\subsubsection{Algebraic}\label{section:algebraicresults}

We now turn to the statements of our main algebraic asymptotic results. Let
\begin{multline*}
N_{m,n}(x):=\#\{\text{central simple algebras $A/k$ of dimension $n^2$ of the form}\\  \text{$A=\MM(r,D)$, where $\dim(D)=d^2$ for some $d\mid m$,  and $\abs{\disc(A)}\le x$}\}.
\end{multline*}
We now state our first algebraic asymptotic counting result.

\begin{thm}\label{thm:area1}
If $N(x)$ denotes the number of division algebras $D/k$ of dimension $n^2$ with $\abs{\disc(D)} \leq x$ and $\ell$ is the smallest prime divisor of $n$, then
\[ N(x) = \sum_{m \mid n} \mu(n/m) N_{m,n}(x). \]
Moreover, there is a constant $\delta_n > 0$, which may depend on $k$, so that
\begin{equation}\label{eq:Nxasymptotic}
N(x) = (\delta_n+o(1)) x^{\frac{1}{n^2(1-1/\ell)}} (\log{x})^{\ell-2} \quad \text{ as } x\to\infty.
\end{equation}
\end{thm}

The key component of this proof is a classical Tauberian theorem of Delange, which allows us to precisely estimate $N_{m, n}(x)$ provided that we understand the analytic behavior of its associated Dirichlet series. We next provide an asymptotic count of the quadratic extensions that embed in a fixed quaternion algebra over a fixed field $k$.

\begin{thm}\label{thm:countingquads}
Fix a number field $k$ and a quaternion algebra $B/k$. The number of quadratic extensions $L/k$ which embed into $B$ and satisfy $\abs{\Delta_{L/k}} \leq x$ is asymptotic to $c_{k,B} x$ as $x\to\infty$, where $c_{k,B} > 0$. Moreover, if $\kappa_k$ is the residue at $s=1$ of $\zeta_k(s)$, $r_2$ is the number of pairs of complex embeddings of $k$, and $r_B'$ is the number of places of $k$ (both finite and infinite) that ramify in $B$, then $c_{k,B} \geq \frac{1}{2^{r_B'+r_2}} \frac{\kappa_k}{\zeta_k(2)}$.
\end{thm}

The proof of this result stems from a powerful theorem of Wood \cite{wood10}, which allows us to model the splitting of finitely many primes as mutually independent events, over the class of random extensions of $k$. Our final result provides an asymptotic count of the quaternion algebras over $k$ with a specified finite collection of maximal subfields. Note that we require some conditions on the collection of subfields as some selections might not have any algebra that contains them as maximal subfields.

\begin{thm}\label{thm:quatembedding}
Fix a number field $k$, fix quadratic extensions $L_1, L_2, \dots, L_r/k$, let $L$ be the compositum of the $L_i$, and suppose that $[L:k] = 2^r$. Then the number of quaternion algebras $B/k$ with $\abs{\disc(B)}$ less than $x$ and which admit embeddings of all of the $L_i$ is $\sim \delta \cdot x^{1/2}/(\log{x})^{1-\frac{1}{2^r}}$ as $x\to\infty$. Here $\delta$ is a positive constant explicitly given in the proof and depending only on the $L_i$ and $k$.
\end{thm}

In this proof, we make use of the well-developed theory of sums of nonnegative multiplicative functions due to Wirsing in order to obtain a precise asymptotic for our counting function. We highlight an explicit value of the constant $\delta$ in the case where $r=1$ below.

\begin{ex0}\label{ex:quatembeddingexample}
When $r=1$, the expression for $\delta$ can be put in compact form. We find that the number of quaternion algebras $B/k$ that admit an embedding of a fixed quadratic extension $L/k$ is
\[ \sim 2^{r_1'-\frac{1}{2}} \left(\frac{\kappa}{L(1,\chi)}\right)^{1/2} \cdot \prod_{\substack{\pp \text{ finite}\\ \pp\text{ not split}}} \left(1-\frac{1}{\abs{\pp}^{2}}\right)^{1/2} \cdot  \prod_{\substack{\pp \text{ finite}\\ \pp\text{ ramified}}} \left(1+\frac{1}{\abs{\pp}}\right)^{1/2} \cdot \frac{x^{1/2}}{(\log{x})^{1/2}} \quad \text{ as } x\to\infty. \]	
Here $\kappa$ is the residue at $s=1$ of $\zeta_k(s)$, and $L(1,\chi)$ is the value at $s=1$ of the nontrivial Artin $L$--function associated to the extension $L/k$.
\end{ex0}

\subsubsection{Geometric}

The above algebraic counting results have geometric companions. We briefly supply some additional, independent motivation before stating our geometric counting applications. Basic problems like counting arithmetic manifolds of a bounded volume modeled on a fixed symmetric space involve two distinct mechanisms for growth: the growth rate coming from a fixed commensurability class and the growth rate of the number of distinct commensurability classes. Several papers have been written on the growth rate of (arithmetic) lattices in a fixed Lie group and also manifolds modeled on a fixed symmetric space; see \cite{Bel07}, \cite{BGLS}, \cite{BL12}, \cite{BGLM02}, \cite{GG08}, \cite{GLNP04}, \cite{GLP04}, \cite{Gol12}, \cite{LL08}, \cite{LS04}, and \cite{LN04}. Our counting results focus on counting commensurability classes of manifolds with some prescribed features, or counting commensurability classes of geodesics or totally geodesic submanifolds in a fixed manifold.





Our first two results provide upper bounds for the number of commensurability classes of arithmetic hyperbolic 2-- or 3--manifolds with a fixed trace field. In the statement of these results, the volume $V_\mathcal{C}$ of a commensurability class $\mathcal{C}$ is the minimum volume achieved by its members. That this volume is achieved in a commensurability class follows from Borel \cite{borel-commensurability}. Belolipetsky \cite[\S 4]{Bel07} gave a polynomial upper bound for the number of commensurability classes of irreducible arithmetic lattices in a fixed isotypic semisimple Lie group $H$ arising from a fixed number field $k$ under the assumption that the simple factors are not of type $\mathrm{A}_1$. When $H$ has type $\mathrm{A}_1$ factors (i.e.~$H$ is isogenous with $\SL(2,\R)^a \times \SL(2,\C)^b$), Belolipetsky--Gelander--Lubotzky--Shalev \cite[\S 3]{BGLS} gave a polynomial upper bound for the number of commensurability classes of arithmetic lattices in $H$ arising from a fixed number field $k$. Varying $k$, both \cite{Bel07}, \cite{BGLS} provided super-polynomial upper bounds for the number of classes without restriction on $k$. Our first results give upper bounds for the number of commensurability classes of arithmetic hyperbolic $2$-- or 3--manifolds with trace field $k$. These upper bounds are explicit refinements of the upper bounds from \cite{BGLS} in these cases.


\begin{cor}\label{cor:area1cor2}
Let $k$ be a totally real number field of degree $n_k$ and let $N_k(V)$ be the number of commensurability classes $\scrC$ of arithmetic hyperbolic 2--orbifolds with trace field $k$ and $V_\scrC\leq V$. Then $N_k(V)\ll \frac{\kappa 2^{n_k-1} V^{130}}{\zeta_k(2)}$ for sufficiently large $V$, where $\zeta_k(s)$ is the Dedekind zeta function of $k$ and $\kappa$ is the residue of $\zeta_k(s)$ at $s=1$.
\end{cor}


\begin{cor}\label{cor:area1cor1}
Let $k$ be a number field of degree $n_k$ with a unique complex place and let $N_k(V)$ be the number of commensurability classes $\scrC$ of arithmetic hyperbolic $3$--orbifolds with trace field $k$ and $V_\scrC\leq V$. Then $N_k(V)\ll \frac{\kappa 2^{n_k-3} V^7}{\zeta_k(2)}$ for sufficiently large $V$, where $\zeta_k(s)$ is the Dedekind zeta function of $k$ and $\kappa$ is the residue of $\zeta_k(s)$ at $s=1$.
\end{cor}

There is one commensurability class of non-compact arithmetic hyperbolic $2$--manifolds and for non-compact arithmetic 3--manifolds, the commensurability classes are in bijection with the imaginary quadratic number fields. In particular, for a fixed $k$, they do not affect the growth of $N_k(V)$.

We say (complex) geodesic lengths $\ell_1,\ell_2$ are \textbf{rationally inequivalent} if $\ell_1/\ell_2 \notin \Q$. Our next result provides a lower bound for the number of rationally inequivalent geodesics of bounded length.

\begin{cor}\label{cor:GeoCount}
Let $M$ be an arithmetic hyperbolic 2--manifold (resp., 3--manifold) of covolume $V$ with invariant trace field $k$ and invariant quaternion algebra $B$. Then for sufficiently large $V$ and $x$, $M$ contains at least $\brac{ \frac{\kappa_k}{2} \pr{ \frac{3}{\pi^2}}^{n_k}}x$ rationally inequivalent geodesics of length at most $e^{cV}x^{n_k}$ where $c$ is an absolute, effectively computable constant.
\end{cor}

An alternative form of this inequality is that there are at least $\brac{\frac{\kappa_k}{2}\pr{\frac{3}{\pi^2}}^{n_k} e^{-cV/n_k}} \ell^{1/n_k}$ rationally inequivalent geodesics of length at most $\ell$ provided that $V$ and $\ell$ are sufficiently large. By Huber \cite{Huber}, Margulis \cite{Margulis}, the asymptotic growth rate for the number of primitive geodesics of length at most $\ell$ is $\frac{e^{h\ell}}{h\ell}$ where $h$ is the entropy of the geodesic flow.

Our final geometric counting result provides a lower bound for the growth rate of incommensurable totally geodesic surfaces of bounded area in an arithmetic hyperbolic 3--manifold that contains at least one totally geodesic surface.

\begin{thm}\label{thm:infinitelymanytotallygeodesics}
Let $M=\mathbf{H}^3/\Gamma$ be an arithmetic hyperbolic $3$--orbifold of volume $V$ with invariant trace field $k$ and invariant quaternion algebra $B$. If $M$ contains a totally geodesic surface, then for sufficiently large $x$, $M$ contains at least $\brac{c(k)\disc(B)^{1/2}}x/\log(x)^{1/2}$ pairwise incommensurable totally geodesic surfaces with area at most $\brac{2\pi^2e^{cV}}x$. Here $c(k)$ is a constant depending only on $k$ and $c$ is an absolute, effectively computable constant.
\end{thm}


Theorem \ref{thm:infinitelymanytotallygeodesics} in tandem with Theorem \ref{EffectiveMcReid} gives an estimate on the number of surfaces needed to distinguish a pair of incommensurable, arithmetic hyperbolic 3--manifolds with a totally geodesic surface. We prove that if an arithmetic hyperbolic 3--manifold contains a totally geodesic surface then in fact it contains a totally geodesic surface with area bounded above by data from the manifold (see Proposition \ref{proposition:ggsprop}).

\subsection{Layout}

In \S 2 we introduce some of the basic concepts, terms, and objects for the paper. In \S 3 we prove the main algebraic counting results. In \S 4 we prove the main geometric counting results. In \S 5 we prove the effective results on geodesic lengths while in \S 6 we prove the results involving surfaces, including the asymptotic results on incommensurable, totally geodesic surfaces.


\paragraph*{\textbf{Acknowledgements.}}

The authors thank Jayadev Athreya, Richard Canary, Ted Chinburg, Britain Cox, Peter Doyle, Daniel Fiorilli, Tsachik Gelander, Grant Lakeland, Chris Leininger, Jeff Meyer, Nick Miller, Gopal Prasad, Alan Reid, and Matthew Stover for conversations on the material in this article. We also thank the anonymous referee for detailed comments on an earlier version that corrected a mistake in Theorem 1.1 and helped improve the exposition. BL was partially supported by NSF RTG grant DMS-1045119 and an NSF Mathematical Sciences Postdoctoral Fellowship. DBM was partially supported by NSF grant DMS-1105710. PP was partially supported by NSF grant DMS-1402268. LT was partially supported NSF VIGRE grant DMS-0738586 and by an AMS Simons Travel Grant.

\section{Background}\label{Background}

\paragraph*{\textbf{Notation}}

The following notation is utilized throughout this article.

\begin{itemize}
\item $\N,\Z,\Q,\R,\C$ are the natural numbers, integers, and rational, real, and complex fields. $\varphi(n)$ is the Euler totient function, $\mu(n)$ is the M\"{o}bius function, and $\log(x)$ is the natural logarithm function.
\item $k$ is a number field and $L/k$ is a finite extension. $\calO_k,\calO_k^\ast,\calO_k^1$ are the ring of integers, the group of units, and the group of norm 1 elements, respectively. $\widehat{k}$ is the Galois closure of $k$ and $k^+$ is the maximal, totally real subfield of $k$. $d_k$ is the absolute discriminant, $\Reg_k$ is regulator for $k$, and $h_k$ is the class number of $k$
\item $\Pp_k$ is the set of places/primes of $k$. For $\pp \in \Pp_k$, we denote the norm by $\abs{\pp}$ and the associated valuation by $\abs{\cdot}_\pp$. For a place $\mathfrak{P} \in \Pp_L$ residing over a place $\pp \in \Pp_k$, we denote this by $\mathfrak{P} \mid \pp$ or simply $\mathfrak{P} \mid_\pp$. Occasionally, $\mathfrak{P} \mid_k$ will denote the prime $\pp \in \Pp_k$ that $\mathfrak{P}$ is over.
\item $n_k$ is the degree of $k/\Q$. $r_1(k)$ and $r_2(k)$, or simply $r_1,r_2$, are the number of real and complex places.
\item $\mathfrak{f}_{L/k}$ is the conductor. $\Delta_{L/k}$ is the relative discriminant of an extension $L/k$. $\zeta_k(s)$ is the Dedekind $\zeta$--function and $\kappa_k$ is the residue of the pole of $\zeta_k$ at $s=1$. $J_k$ is the id\`ele group for $k$.
\item $D/k$ is a division algebra over $k$. $B/k$ is a quaternion algebra over $k$. $A/k$ is a central simple algebra over $k$ with norm $\nr$, group of invertible elements $A^\times$, group of norm 1 elements $A^1$, and discriminant $\disc(A)$.
\item $\Ram(A) \su \Pp_k$ is the set of ramified places of $A$, $\Ram_f(A)$ is the set of finite ramified places of $A$, $\Ram_\iny(A)$ is the set of infinite ramified places of $A$, and $r_A=\abs{\Ram_\iny(A)}$. 
\item For a quaternion algebra $B$, $k_B$ is the maximal abelian extension of $k$ which has $2$--elementary Galois group, is unramified outside of the real places in $\Ram(B)$, and in which every $\pp \in \Ram_f(B)$ splits completely.
\item $\calO, \calE,\calD$ are $\mathcal{O}_k$--orders in a central simple $k$--algebra. For each place $\pp \in \Pp_k$, $\calO_\pp$ is the completion of $\calO$ at $\pp$.
\item $\mathbf{H}^2,\mathbf{H}^3$ are real hyperbolic 2-- and 3--spaces. $M$ is an arithmetic hyperbolic 2-- or 3--manifold and $\Gamma= \pi_1(M) < \PSL(2,\R)$ or $\PSL(2,\C)$ is the associated arithmetic lattice. When $\Gamma = P\rho(\calO^1)$, we write $\Gamma = \Gamma_\calO$.
\item $c,C$ and variously decorated versions are constants. We interchangeably use the Landau ``Big Oh'' notation, $f =O(g)$, and the Vinogradov notation, $f \ll g$, when there exists a constant $C > 0$ such that $|f| \leq C|g|$. $f \sim g$ if $\lim_{x \rightarrow \infty} \frac{f(x)}{g(x)} = 1$ and $f = o(g)$ if $\lim_{x \rightarrow \infty} \frac{f(x)}{g(x)} = 0.$
\end{itemize}


\subsection{Algebraic}

We refer the reader to \cite{CF}, \cite{Lang-ANT}, \cite{Marcus}, \cite{Pierce}, and \cite{Reiner} for the below material.

\subsubsection{Central simple algebras}
\quad

One main algebraic requisite for later discussion is the theory of central simple algebras $A/k$ and their orders. By the Artin--Wedderburn Structure Theorem \cite[p.~49]{Pierce}, every such $A$ is isomorphic to a matrix algebra over a division algebra $A=\MM(r,D)$ where $r$ and $D$ are uniquely determined. We require the following theorem in this paper.

\begin{thm}\label{lem:allDAs}
Let $k$ be a number field. Let $S$ be a finite collection of primes of $k$ consisting of finite primes and real infinite places. Suppose that for each $\pp \in S$ we are given a reduced fraction $\frac{a_\pp}{m_\pp} \in \Q \cap (0,1)$ such that
\begin{enumerate}
\item
$m_\pp > 1$ and $a_\pp >0$,
\item
$\frac{a_\pp}{m_\pp}=\frac{1}{2}$ whenever $\pp$ is real,
\item
$\sum_{\pp \in S} \frac{a_{\pp}}{m_{\pp}} \in \Z$.
\end{enumerate}
There is a unique division algebra $D/k$ possessing $S$ as its set of ramified primes and with Hasse invariants $\frac{a_{\pp}}{m_{\pp}}$ for $\pp \in S$. Conversely, every division algebra $D/k$ arises in this way. The dimension of $D$ is $n^2$, where $n := \lcm_{\pp \in S}[m_\pp]$. The discriminant of $D$ is the modulus of $k$ given by
\[ \disc(D) = \prod_{\substack{\pp \in S \\ \pp \text{ real}}}\pp \prod_{\substack{\pp \in S \\ \pp \text{ finite}}} \pp^{n^2(1-\frac{1}{m_\pp})}.  \]
\end{thm}

This theorem is a consequence of the Albert--Brauer--Hasse--Noether theorem (see for instance \cite[\S 18.4]{Pierce}) and more generally, the short exact sequence of Brauer groups appearing in local class field theory. Moreover, one can show that
\begin{align*}
\disc(A) &= \prod_{\substack{\pp\text{ real} \\ \pp\mid \disc(D)}} \pp \cdot \bigg(\prod_{\substack{\pp\text{ finite} \\ \pp\mid \disc(D)}} \pp\bigg)^{r^2} = \prod_{\substack{\pp \in S \\ \pp \text{ real}}}\pp \prod_{\substack{\pp \in S \\ \pp \text{ finite}}} \pp^{n^2(1-\frac{1}{m_\pp})} \\
\end{align*}
when $A = \MM(r,D)$. Thus,
\begin{equation}\label{eq:discAexpr}
\abs{\disc(A)} = \abs{\disc(D)}^{r^2} = \prod_{\substack{\pp \in S \\ \pp \text{ finite}}} \abs{\pp}^{n^2(1-\frac{1}{m_\pp})}.
\end{equation}
By Theorem \ref{lem:allDAs}, $D$ corresponds to certain Hasse invariants $a_{\pp}/m_{\pp}$ for $\pp \in \Pp$, and the dimension of $D$ over $k$ is $d^2$ for $d = \lcm_{\pp \in\Pp}[m_\pp]$. In the future, the Hasse invariants of $D$ will also be referred to as the \textbf{Hasse invariants of $A$}. If the dimension of $A$ over $k$ is $n^2$, then $r^2 d^2 = n^2$.
\subsubsection{Parametrizing maximal orders}

In what follows, $k$ will be a fixed number field and $B/k$ a fixed quaternion algebra. Our exposition follows \S 3--4 of \cite{linowitz-selectivity}. We refer the reader to \cite{Reiner} for a general treatment on orders.

Let $J_k$ (respectively $J_B$) denote the id\`ele group of $k$ (respectively $B$). In this context, the id\`ele group $J_B$ acts on the set of maximal orders of $B$ as follows. Given $\tilde{x}\in J_B$ and $\calO$ a maximal order of $B$ we define $\tilde{x}\calO\tilde{x}^{-1}$ to be the unique maximal order of $B$ with the property that for every finite prime $\pp$ of $k$, its completion at $\pp$ is equal to $x_\pp\calO_\pp x_\pp^{-1}$ (existence and uniqueness follow from the local-to-global correspondence for orders). With this action we see that the set of maximal orders corresponds to the coset space $J_B/\mathfrak{N}(\calO)$, where $\mathfrak{N}(\calO)=J_B\cap \prod_{\pp} N_{B_\pp^*}(\calO_\pp)$ and $N_{B_\pp^*}(\calO_\pp)$ is the normalizer in $B_\pp^*$ of $\calO_\pp^*$. The isomorphism classes  of maximal orders of $B$ (which by the Skolem--Noether theorem \cite[p.~230]{Pierce} coincide with conjugacy classes) thus correspond to points in the double coset space $B^*\backslash J_B/\mathfrak{N}(\calO)$. The reduced norm $\nr(\cdot)$ induces a bijection \cite[Thm 4.1]{linowitz-selectivity} between the latter double coset space and $k^*\backslash J_k/\nr(\mathfrak{N}(\calO))\cong J_k/k^*\nr(\mathfrak{N}(\calO))$. The latter group is finite and, as $J_k^2\subset \nr(\mathfrak{N}(\calO))$, is of exponent $2$. Hence, there exists an integer $m\geq 1$ such that the number of isomorphism classes of maximal orders is equal to $2^m$ and $J_k/k^*\nr(\mathfrak{N}(\calO))\cong \left(\Z/2\Z\right)^m$.

We now parameterize the maximal orders of $B$. Let $\pp_1,\dots,\pp_m$ be a set of primes of $k$ such that $B_{\pp_i}\cong \MM(2,k_{\pp_i})$ for all $i$ and such that the cosets of $J_k/k^*\nr(\mathfrak{N}(\calO))$  defined by the elements
\[ \set{e_{\pp_i}=(1,\dots,1,\pi_{\pp_i},1,\dots) }_{i=1}^m \]
form a generating set. For each prime $\pp_i$, let $\delta_i=\diag(\pi_{\pp_i},1)$ and $\calO_{\pp_i}^\prime=\delta_i\calO_{\pp_i}\delta_i^{-1}$. Given $\gamma=(\gamma_i)\in\left(\Z/2\Z\right)^m$, we define a maximal order $\calO^{\gamma}$ via the local-to-global correspondence:
\begin{displaymath}
\calO_\pp^\gamma = \left\{ \begin{array}{ll}
\calO_{\pp_i} & \textrm{if $\pp=\pp_i$ and $\gamma_i=0$}\\
\calO_{\pp_i}^\prime & \textrm{if $\pp=\pp_i$ and $\gamma_i=1$}\\
\calO_\pp & \textrm{otherwise.}
\end{array}\right.
\end{displaymath}
By \cite[Prop 4.1]{linowitz-selectivity}, every maximal order of $B$ is conjugate to one of the orders defined above. Henceforth we will refer to this as a \textbf{parameterization of the maximal orders of $B$ relative to $\calO$}.

Let $L/k$ be a quadratic extension and $k_B$ be the class field corresponding to $J_k/k^*\nr(\mathfrak{N}(\calO))$ by class field theory. Alternatively, $k_B$ can be characterized as the maximal abelian extension of $k$ which has $2$--elementary Galois group, is unramified outside of the real places in $\Ram(B)$ and in which every finite prime of $\Ram(B)$ splits completely. The following lemma appears as \cite[Lemma 3.7]{linowitz-selectivity}:

\begin{lem}\label{lem:generatingsetlemma}
Let the notation be as above.
\begin{enumerate}
\item If $L\not\subset k_B$ then there exists a generating set $\{e_{\pp_i}\}$ of $J_k/k^*\nr(\mathfrak{N}(\calO))$ in which all of the $\pp_i$  split in $L/k$.
\item If $L\subset k_B$ and $\qq$ is any prime of $k$ which is inert in $L/k$ then there exists a generating set $\{e_{\pp_i}\}$ of $J_k/k^*\nr(\mathfrak{N}(\calO))$ in which $\pp_1=\qq$ and $\pp_2,\dots,\pp_m$ all split in $L/k$.
\end{enumerate}	
\end{lem}

We conclude this section with a technical result which we will utilize in the proof of Theorem \ref{thm:effectiveCHLR}.

\begin{prop}\label{proposition:selectivecase}
Let $\calE,\calD$ be maximal orders of $B$ and suppose that $u\in\calE^1$ with $u\notin k$ and set $L= k(u)$.  Then there exists an absolute constant $C_1>0$ and a positive integer $n\leq d_L^{C_1}$ such that $\calD$ admits an embedding of $\calO_k[u^n]$.
\end{prop}

In the proof of Proposition \ref{proposition:selectivecase}, we require the following lemma, which is an immediate consequence of the bound on the least prime ideal in the Chebotarev density theorem \cite{effectiveCDT}.

\begin{lem}\label{lem:smallestinert}
If $L/k$ is a quadratic extension, then there exists an absolute, effectively computable constant $C_1$ such that there exists a prime of $k$ which is inert in $L/k$ and has norm less than $d_L^{C_1}$.
\end{lem}

\begin{proof}[Proof of Proposition \ref{proposition:selectivecase}]
If $L\not\subset k_B$, then the selectivity theorem of Chinburg--Friedman \cite{Chinburg-Friedman-selectivity} (see also \cite{linowitz-selectivity}) shows that $\calD$ admits an embedding of $\calO_k[u]$, hence we may take $n=1$. Suppose now that $L\subset k_B$ and let $\qq$ be a prime of $k$ of smallest norm which is inert in $L/k$. Let  $\{ e_{\pp_i} \}$ be the set of representatives of $J_k/k^*\nr(\mathfrak{N}(\calO))$ from Lemma \ref{lem:generatingsetlemma}(ii) (in which $\pp_1=\qq$ and $\pp_2,\dots,\pp_m$ all split in $L/k$). We claim that $\qq$ does not ramify in $B$. Indeed, suppose that $\qq$ ramified in $B$. By our characterization of $k_B$, it would follow that $\qq$ would split completely in $k_B$ and hence in $L$ as $L\subset k_B$. However, this observation contradicts the fact that $\qq$ is inert in $L/k$, proving our claim.

For each $i=2,\dots,m$, we have an $k_{\pp_i}$--isomorphism $f_{\pp_i}\colon B_{\pp_i}\to\MM(2,k_{\pp_i})$ such that $f_{\pp_i}(L)\subset  \begin{pmatrix} k_{\pp_i} & 0 \\ 0 & k_{\pp_i} \end{pmatrix}$. Consequently, $f_{\pp_i}(\calO_L)\subset  \begin{pmatrix} \calO_{k_{\pp_i}} & 0 \\ 0 & \mathcal O_{k_{\pp_i}} \end{pmatrix}$, and so $\calO_k[u]$ is contained in two adjacent vertices in the tree of maximal orders of $\MM(2,k_{\pp_i})$. Upon conjugating $\calE$ if necessary, we may assume that $\{\calE^{\gamma}\}$ is a parameterization of the maximal orders of $B$ relative to $\calE$. Additionally, we have $u\in\calE^{\gamma}_{\pp_i}$ for all $\gamma$ and $i=2,\dots,m$. By construction $\calE_{\qq},\calE_{\qq}^\prime$ are adjacent in the tree of maximal orders of $\MM(2,k_{\qq})$. By \cite[p.~340]{MR}, we have $[\calE_{\qq}^1:\calE_{\qq}^1\cap {\calE_{\qq}^\prime}^1 ]=\abs{\qq}(\abs{\qq}+1$). Setting $n=\abs{\qq}(\abs{\qq}+1)$, we have shown that $u^n\in\calE^{\gamma}_{\pp_i}$ for all $\gamma$ and $1\leq i\leq m$. As $\calE^{\gamma}_{\pp}=\calE_{\pp}$ for all primes $\pp \notin \{\pp_1,\dots,\pp_m\}$, we conclude that $u^n\in \calE^{\gamma}$ for all $\gamma$. As all maximal orders of $B$ are conjugate to one of the $\calE^{\gamma}$, the proposition follows from Lemma \ref{lem:smallestinert}.
\end{proof}

\subsection{Geometric}

We refer the reader to Maclachlan--Reid \cite{MR} for a thorough treatment of this material.

\subsubsection{Hyperbolic geometry}

Hyperbolic $n$--space $\mathbf{H}^n$ is the real rank one symmetric space associated to the real simple Lie group $\SO(n,1)$. We identify the group of orientation preserving isometries of $\mathbf{H}^2,\mathbf{H}^3$ with $\PSL(2,\R),\PSL(2,\C)$, respectively. We view $\mathbf{H}^2,\mathbf{H}^3$ as the symmetric spaces $\mathbf{H}^2 = \PSL(2,\R)/\SO(2)$, $\mathbf{H}^3 = \PSL(2,\C)/\SU(2)$. Isometries of $\mathbf{H}^2,\mathbf{H}^3$ split into three classes depending on the trace of the element. An isometry $\gamma \in \PSL(2,\C)$ is \textbf{elliptic} if $\Tr(\gamma) \in (-2,2)$, \textbf{parabolic} if $\abs{\Tr(\gamma)} = 2$, and \textbf{hyperbolic} if $\abs{\Tr(\gamma)} > 2$. 

We will sometimes refer to lattices $\Gamma$ in $\PSL(2,\R)$ or $\PSL(2,\C)$ as \textbf{Fuchsian} of \textbf{Kleinian} groups. Given a lattice $\Gamma$ in $\PSL(2,\R),\PSL(2,\C)$, the associated quotient $M = \textbf{H}^2/\Gamma,\textbf{H}^3/\Gamma$ is a complete finite volume hyperbolic $2$-- or $3$--orbifold. We state an inequality of Gelander \cite[Thm 1.7]{Gelander} involving the volume of a complete, finite volume hyperbolic 3--manifold $M$ and the rank of the fundamental group $\pi_1(M) = \Gamma$.


\begin{thm}[Gelander]\label{RankVolumeInequality}
There exists a constant $C$ such that if $M$ is a complete, finite volume hyperbolic 3--manifold of volume $V$ and rank $r$ fundamental group, then $r \leq CV$.
\end{thm}

\subsubsection{Arithmetic hyperbolic manifolds}

Let $k$ be a totally real field with real places $\pp_1,\dots,\pp_{r_1}$. Fix a real place of $k$ which, reordering if necessary, we denote by $\pp_1$. We select a quaternion algebra $B/k$ with the property that $\pp_j \in \Ram(B)$ if and only if $j>1$. In particular, $B_{\pp_1} \cong \MM(2,\R)$ and $B_{\pp_j} \cong \mathbb{H}$ for $j>1$, where $\mathbb{H}$ is the quaternions over $\R$. Under the first isomorphism, the group of norm one elements $B^1$ maps into $\SL(2,\R)$. Selecting a maximal order $\calO \subset B$, the image of $\calO^1$ in $\SL(2,\R)$ and the image of the projection $P\calO^1$ of $\calO^1$ to $\PSL(2,\R)$ are both lattices. We say $\Gamma < \SL(2,\R)$ or $\PSL(2,\R)$ is an \textbf{arithmetic lattice} if $\Gamma$ is commensurable in the wide sense with $\calO^1$ or $P\calO^1$ for some totally real number field $k$, quaternion algebra $B/k$, and maximal order $\calO \subset B$ as above. We use $\Gamma_\calO$ to denote $\calO^1$ and we say a lattice $\Gamma<\SL(2,\R)$ is \textbf{derived from a quaternion algebra} if $\Gamma < \Gamma_\calO$ for some $k,B,\calO$ as above. 

The construction of arithmetic lattices in $\PSL(2,\C)$ is similar. Let $k$ be a number field with a unique complex place $\pp_1$ and real places $\pp_2,\dots,\pp_{r_1+1}$ and let $B/k$ be quaternion algebra such that $\pp_j \in \Ram(B)$ if and only if $j>1$. Fixing an isomorphism $B_{\pp_1} \cong \MM(2,\C)$ and a maximal order $\calO \subset B$, the groups $\calO^1,P\calO^1$ are lattices in $\SL(2,\C), \PSL(2,\C)$. Any lattice $\Gamma < \SL(2,\C),\PSL(2,\C)$ that is commensurable in the wide sense with $\calO^1, P\calO^1$, for some $k,B,\calO$ as above, will be called an \textbf{arithmetic lattice} in $\SL(2,\C), \PSL(2,\C)$. We say a lattice $\Gamma<\SL(2,\C)$ is \textbf{derived from a quaternion algebra} if $\Gamma < \Gamma_\calO$ for some $k,B,\calO$ as above. Finally, if $B$ is a division algebra, $\calO^1, P\calO^1$ are cocompact.


Given two arithmetic lattices $\Gamma_1,\Gamma_2$ arising from $(k_j,B_j)$, $\Gamma_1,\Gamma_2$ will be commensurable in the wide sense if and only if $k_1 \cong k_2$ and $B_1 \cong B_2$ (see \cite[Thm 8.4.1]{MR}). We will make use of this fact throughout the remainder of this article.

\begin{thm}\label{CommensurabilityInvariants}
Let $\Gamma_1,\Gamma_2$ be arithmetic lattices in $\PSL(2,\R)$ or $\PSL(2,\C)$ with arithmetic data $(k_1,B_1),(k_2,B_2)$, respectively. Then $\Gamma_1,\Gamma_2$ are commensurable in the wide sense if and only if $k = k_1 \cong k_2$ and $B_1 \cong B_2$ as $k$--algebras.
\end{thm}

Theorem \ref{CommensurabilityInvariants} was proven by Takeuchi \cite{takeuchi} for Fuchsian groups and for Kleinian groups by Macbeath \cite{Macbeath}, Reid \cite{ReidInv}. We say that $M$ is an \textbf{arithmetic hyperbolic 2-- or 3--orbifold} if the orbifold fundamental group of $M$ is an arithmetic lattice in $\PSL(2,\R)$ or $\PSL(2,\C)$. In this case, $\pi_1(M)=\Gamma$ is commensurable with $P\calO^1$ for some $k,B,\calO$ as above. We call $k$ the \textbf{invariant trace field/trace field} and $B$ the \textbf{invariant quaternion algebra/quaternion algebra} of $M$.

\subsubsection{Geodesics and quadratic subfields}

Let $M$ be an arithmetic hyperbolic 2-- or 3--orbifold arising from $(k,B)$ with orbifold fundamental group $\Gamma < \PSL(2,\R)$ or $\PSL(2,\C)$. The closed geodesics $c_\gamma\colon S^1 \to M$ on $M$ are in bijection with the $\Gamma$--conjugacy classes $[\gamma]_\Gamma$ of hyperbolic elements $\gamma$ in $\Gamma$. The roots of the characteristic polynomial $p_\gamma(t)$ are given by the eigenvalues of $\gamma$ and the associated (complex) geodesic length $\ell(c_\gamma)$ is given by (see \cite[p.~372]{MR})
\begin{equation}\label{TraceLengthFormula}
\cosh\pr{\frac{\ell(c_\gamma)}{2}} = \pm \frac{\Tr(\gamma)}{2}.
\end{equation}
When $\Gamma < \PSL(2,\R)$, \eqref{TraceLengthFormula} gives the \textbf{length} of the geodesic associated to $\gamma$. When $\Gamma<\PSL(2,\C)$, \eqref{TraceLengthFormula} gives the \textbf{complex length} of the geodesic associated to $\gamma$. In this case, $\ell(\gamma) = \ell_0(\gamma) + i\theta(\gamma)$ where $\theta(\gamma)$ is the angle of rotation about the geodesic axis associated to $\gamma$ and $\ell_0(\gamma)$ is the length of the geodesic associated to $\gamma$. In particular, when $\Gamma<\PSL(2,\C)$, $\ell(\gamma)$ is the associated complex length and $\ell_0(\gamma)$ is the length. When $\Gamma < \PSL(2,\R)$, $\ell(\gamma)$ will denote the length of the associated geodesic. We denote by $\lambda_\gamma$ the unique eigenvalue of $\gamma$ with $\abs{\lambda_\gamma} > 1$ and note that $\lambda_\gamma \in \calO_{k_\gamma}^1$ when $\Gamma$ is arithmetic. Also, each $\gamma$ determines a maximal subfield $k_\gamma = k(\lambda_\gamma)$ of the quaternion algebra $B$.


\subsubsection{Totally geodesic surfaces}


Asssociated to an arithmetic hyperbolic 3--orbifold is a pair $(L,B)$, where $L$ is a number field with exactly one complex place and $B/L$ is a quaternion algebra that is ramified at all of the real places. If there exists a totally real subfield $k \su L$ with $L/k$ quadratic and $B_0/k$ is a quaternion algebra such that $B_0 \otimes_k L \cong B$, then the pair $(k,B_0)$ will be data for a commensurability class of arithmetic hyperbolic $2$--orbifolds provided that $B_0$ is unramified at the real place $\pp$ under the complex place $\mathfrak{P}$ of $K$. The following appears as \cite[Thm 9.5.5]{MR}.

\begin{thm}\label{TGS-Theorem}
Let $\Gamma$ be an arithmetic lattice in $\PSL(2,\C)$ with arithmetic data $(L,B)$ and suppose that $k$ is a totally real subfield of $L$ with $[L:k]=2$. Suppose $B_0$ is a quaternion algebra over $k$ ramified at all real places of $k$ except at the place under the complex place of $L$. Then $B \cong B_0 \otimes_k L$ if and only if $\Ram_f(B)$ consists of $2r$ places (where $r\geq 0$ ) $\set{\mathfrak{P}_{i,j}}$ where $j$ ranges over $\set{1, \dots, r}$ and $i$ ranges over $\set{1,2}$ and satisfy $\mathfrak{P}_{1,j}\cap \mathcal{O}_k = \mathfrak{P}_{2,i} \cap \mathcal{O}_k = \pp_i$, where $\set{\pp_1,\dots,\pp_r} \subset \Ram_f(B_0)$ with $\Ram_f(B_0)\setminus \set{\pp_1,\dots ,\pp_r}$ consisting of primes in $\mathcal{O}_k$ which are inert or ramified in $L/k$.
\end{thm}


\section{Main tools: Algebraic counting results}\label{MainTools-Algebraic}

We now begin our first main section, where we will establish our algebraic counting results. For the reader interested only in the applications of these results to the rigidity theorems stated in the introduction, the reader can start at Section \ref{GeometricRigiditySection} and refer back to the results from Sections \ref{MainTools-Algebraic} and \ref{Maintools:Geometriccount}.

\subsection{Proof of Theorem \ref{thm:area1}}

Fix a number field $k$ and fix a positive integer $n$. In this section, we estimate the number of division algebras $D/k$ of dimension $n^2$ whose discriminant lies below a large bound $x$.  Our main tool is the following Tauberian theorem of Delange \cite{DelangeOG,Delange}.

\begin{thm}[Delange]\label{DelangeTheorem}
Let
\[ G(s) = \sum_{N=1}^{\infty} a_N N^{-s} \]
be a Dirichlet series satisfying the following conditions for certain real numbers $\rho > 0$ and $\beta > 0$:
\begin{enumerate}
\item each $a_N \geq 0$,
\item $G(s)$ converges for $\Re(s) > \rho$,
\item $G(s)$ can be continued to an analytic function in the closed half-plane $\Re(s) \geq \rho$, except possibly for a singularity at $s=\rho$ itself,
\item there is an open neighborhood of $\rho$, and functions $A(s)$ and $B(s)$ analytic at $s=\rho$, with
\[ G(s) = \frac{A(s)}{(s-\rho)^\beta} + B(s) \]
at every point $s$ in this neighborhood having $\Re(s) > \rho$.
\end{enumerate}
Then as $x\to\infty$,
\[ \sum_{N \leq x} a_N = \left(\frac{A(\rho)}{\rho \Gamma(\beta)}+o(1)\right) x^{\rho} (\log{x})^{\beta-1}. \]
\end{thm}

\begin{rmk}
We allow the possibility that $A(\rho)=0$, in which case the conclusion of Theorem \ref{DelangeTheorem} is that
\[ \sum_{N \leq x} a_N = o(x^\rho (\log{x})^{\beta-1}), \]
as $x\to\infty$. While Delange's theorem is usually stated with the restriction that $A(\rho) \neq 0$, the cases when $A(\rho)=0$ follow with no extra difficulty. For instance, suppose that $\rho$ is the reciprocal of a positive integer, a condition that holds in all of our applications. If $A(\rho)=0$, we can apply the restricted theorem first to $G(s) + \zeta(s/\rho)^{\beta}$, then to $\zeta(s/\rho)^{\beta}$, and then subtract the results to get the assertion we want. If $\rho$ is not the reciprocal of a positive integer, then $\zeta(s/\rho)$ need not be a Dirichlet series itself. However, this argument still works, provided we take as our starting point Delange's original Tauberian theorem \cite{DelangeOG}, which is in terms of Laplace transforms, instead of its consequences for Dirichlet series \cite{Delange}.
\end{rmk}

According to Theorem \ref{lem:allDAs}, a division algebra $D$ over $k$ is uniquely specified by its Hasse invariants (i.e., the set $S=\Ram(D)$ and the choice of fractions $\{a_\pp/m_\pp\}_{\pp \in S}$). Thus, our task is to count the number of ways of choosing these invariants so that the resulting division algebra $D$ has dimension $n^2$ and $\abs{\disc(D)}\le x$. It turns out that this is a difficult problem to attack directly. More natural, from the analytic side, is to first count \emph{all} central simple algebras over $k$ of dimension $n^2$. The count of division algebras can then be obtained by the inclusion-exclusion principle.  We now carefully execute the above approach by introducing the following set:
\begin{multline*}
N_{m,n}(x):=\#\{\text{central simple algebras $A/k$ of dimension $n^2$ of the form}\\  \text{$A=\MM(r,D)$, where $\dim(D)=d^2$ for some $d\mid m$,  and $\abs{\disc(A)}\le x$}\}.
\end{multline*}
The remarks earlier in this paragraph show that in general, $N_{m,n}(x)$ counts the number of choices for Hasse invariants for which $\lcm_{\pp \in S}[m_{\pp}]$ divides $m$ and the product in \eqref{eq:discAexpr} is bounded by $x$. Our key lemma is the following estimate for $N_{m,n}(x)$. Note that the special case of the lemma when $m=n$ provides us with asymptotic behavior for the counting function of all dimension $n^2$ central simple algebras over $k$.

\begin{lem}\label{lem:keylem0}
Let $k/\Q$ be a number field. Let $n > 1$ be an integer, and let $\ell$ be the smallest prime factor of $n$. Let $m$ be a divisor of $n$. Then as $x\to\infty$,
\[ N_{m,n}(x) = (\delta_{m,n} +o(1)) x^{\frac{1}{n^2(1-1/\ell)}} (\log{x})^{\ell-2} \]
for a certain constant $\delta_{m,n}$. If $\ell \nmid m$, then $\delta_{m,n}=0$. Suppose now that $\ell \mid m$. Let $\kappa$ denote the residue at $s=1$ of the Dedekind zeta function $\zeta_k(s)$. If $m$ is odd, then
\begin{multline}
\delta_{m,n} = \frac{\kappa^{\ell-1}}{m(\ell-2)!} \cdot \frac{1}{(n^2(1-1/\ell))^{\ell-2}} \cdot \\ \sum_{\substack{0 \leq j < m\\ \ell \mid j}}\left( \prod_{\pp \text{ finite}} \left(1+\frac{\ell-1}{|\pp|} + \sum_{\substack{m_{\pp} \mid m \\ m_{\pp}> \ell}} \frac{\mu(\frac{m_\pp}{(m_\pp,j)}) \varphi(m_\pp)/\varphi(\frac{m_\pp}{(m_\pp,j)})}{|\pp|^{(1-1/m_\pp)/(1-1/\ell)}} \right)\left(1-\frac{1}{|\pp|}\right)^{\ell-1} \right).\label{eq:deltamdef}
\end{multline}
If $m$ is even, \eqref{eq:deltamdef} needs to be multiplied by $2^{r_1}$, where $r_1$ is the number of real embeddings of $k$.
\end{lem}

Broadly, the proof of the above lemma proceeds as follows. We first set up the count for the algebras and produce corresponding Dirichlet series. We then verify that our Dirichlet series satisfy (i)-(iv) of Theorem \ref{DelangeTheorem}. We complete the proof via Theorem \ref{DelangeTheorem}. The latter two steps comprise the bulk of the work.

\begin{proof} To set up our count, recall that by Theorem \ref{lem:allDAs} a division algebra is determined by its collection of Hasse invariants, the rational numbers $a_{\pp}/m_{\pp}$ for places $\pp \mid \Ram(D)$. Thus, we need only count the number of choices for these local invariants. To that end, we will use the orthogonality relations among the roots of unity in an essential way. For each $d$ dividing $m$, we introduce the set
\[ \farey(m,d):= \set{ 1 \leq k \leq m~:~\frac{k}{m}\text{ has lowest terms denominator $d$}}. \]
For $0 \leq j < m$, let $\zeta_j = e^{2\pi i j/m}$, and consider the formal expression
\[ \frac{1}{m} \pr{\sum_{0 \leq j < m} \brac{\prod_{\pp\text{ finite}} \pr{1 + \sum_{\substack{m_\pp \mid m\\ m_\pp > 1}}\frac{\sum_{a_\pp \in \farey(m,m_\pp)} \zeta_j^{a_\pp}}{\pp^{n^2(1-1/m_\pp)s}}\right) \prod_{\substack{\pp\text{ real} \\ 2\mid m}} \left(1+ \frac{\zeta_j^{m/2}}{\pp^s}}}}, \]
where the conditions on the final product mean that this product appears only when $m$ is even. Expanding, we obtain a formal sum of terms $c_\m/\m^s$, where $\m$ ranges over the moduli of $k$. For $c_{\m}$ to be nonvanishing, it is necessary that every finite prime $\pp$ dividing $\m$ appears to an exponent of the form $n^2 (1-1/m_{\pp})$ for some integer $m_{\pp} > 1$ dividing $m$, and that $\m$ is not divisible by real primes except possibly if $2\mid m$. If these conditions are satisfied, then
\[ c_{\m} = \frac{1}{m}\sum_{0\le j < m} c_{\m}^{(j)}, \quad\text{where}\quad c_{\m}^{(j)} = \sum_{\substack{\{a_{\pp}\}_{\pp \mid \m,~\pp~\text{finite}}\\\text{each } a_{\pp} \in \farey(m,m_\pp)}} \zeta_{j}^{\sum_{\pp\mid \m,\, \text{finite}} a_{\pp} + \sum_{\pp\mid \m,\,\text{real}} \frac{m}{2}}. \]
Writing $a_{\pp}/m = a_{\pp}'/m_{\pp}$, and subtituting in the value of $\zeta_j$, this expression for $c_{\m}^{(j)}$ can be put in the form
\[ \sum_{\substack{\{a_{\pp}'\}_{\pp \mid \m,~\pp~\text{finite}}\\ 1 \le a'_{\pp} \le m_{\pp},~\gcd(a_{\pp}', m_{\pp})=1}} \exp\left(\frac{2\pi i j}{m} \left(m\sum_{\pp \mid \m,~\text{finite}}\frac{a_{\pp}'}{m_{\pp}} + m\sum_{\pp\mid \m,~\text{real}} \frac{1}{2} \right) \right). \]
For each integer $k$, the sum $\sum_{0 \le j < m} \exp(2\pi i jk/m)$ vanishes unless $m \mid k$, in which case the sum takes the value $m$. Now
\[ m\sum_{\substack{\pp \mid \m\\ \pp\text{ finite}}}\frac{a_{\pp}'}{m_{\pp}} + m\sum_{\substack{\pp\mid \m \\ \pp \ \text{real}}} \frac{1}{2} \quad \text{is a multiple of $m$} \Longleftrightarrow \sum_{\substack{\pp \mid \m \\ \pp\text{ finite}}}\frac{a_{\pp}'}{m_{\pp}} + \sum_{\substack{\pp\mid \m \\ \pp \ \text{real}}} \frac{1}{2} \in \Z. \]
Letting $S$ be the set of primes dividing $\m$, we conclude that $c_{\m} = \frac{1}{m} \sum_{0 \le j < m} c_{m}^{(j)}$ counts the number of choices of Hasse invariants for which $\lcm_{\pp \in S}[m_{\pp}]$ divides $m$ and
\[ \bigg(\prod_{\substack{\pp\text{ real} \\ \pp \in S}}\pp\bigg) \bigg(\prod_{\substack{\pp\text{ finite} \\ \pp \in S}} \pp^{n^2(1-1/m_P)}\bigg) = \m. \]
There is a one-to-one correspondence between these choices of Hasse parameters and central simple algebras $A/k$ of dimension $n^2$ of the form $\MM(r,D)$, where $\dim(D) = d^2$ for some $d \mid m$, and $\disc(A) = \m$. Since $d=\lcm[m_{\pp}]$ and $r=n/d$, we can view the coefficients $c_{\m}$ as counting these central simple algebras. Thus,
\[ N_{m,n}(x) = \sum_{|\m| \le x} c_\m. \]

In order to apply Delange's theorem, Theorem \ref{DelangeTheorem}, we need Dirichlet series. We obtain the needed series by simply replacing the primes $\pp$ with their norms $\abs{\pp}$ in the above products; here we set $|\pp|=1$ when $\pp$ is real. For $j=0, 1, 2, \dots, m-1$, the Dirichlet series $G_j(s)$ is given by the following product expansion:
\begin{equation}\label{eq:gjsdef}
G_j(s) = \prod_{\pp\text{ finite}} \left(1 + \sum_{\substack{m_\pp \mid m\\ m_\pp > 1}}\frac{\sum_{a_\pp \in \farey(m,m_\pp)} \zeta_j^{a_\pp}}{|\pp|^{n^2(1-1/m_\pp)s}}\right) \prod_{\substack{\pp\text{ real} \\ 2\mid m}} \left(1+ \frac{\zeta_j^{m/2}}{\abs{\pp}^s}\right).
\end{equation}
If we then set
\[ G(s) = \frac{1}{m} \sum_{0 \leq j < m} G_j(s), \]
the coefficient of $N^{-s}$ in $G(s)$ is precisely $\sum_{\abs{\m}=N} c_{\m}$. Hence, $N_{m,n}(x)$ is precisely the partial sum up to $x$ of the coefficients of $G(s)$.

We now establish that our Dirichlet series satisfy the conditions of Theorem \ref{DelangeTheorem}.

\textbf{Claim.} $G(s)$ satisfies (i)-(iv) of Theorem \ref{DelangeTheorem}.

\begin{proof}[Proof of Claim]
Since the coefficients of $G(s)$ count central simple algebras, their non-negativity is obvious. This shows that condition (i) of Delange's theorem is satisfied. To verify conditions (ii)-(iv) for $G(s)$, it suffices to verify that they hold for each individual $G_j(s)$. We will show this with
\begin{equation}\label{eq:rhobetachoice}
\rho = \frac{1}{n^2(1-1/\ell)} \quad\text{and}\quad \beta = \ell-1.
\end{equation}
Proceeding further requires a careful study of the product definition \eqref{eq:gjsdef} of $G_j(s)$. First, we deal with the product over real primes, which is present only when $2\mid m$. By our convention that real primes have norm $1$,
\[ \left(1+ \frac{\zeta_j^{m/2}}{\abs{\pp}^s}\right) = (1+\zeta_j^{m/2})^{r_1}; \]
in particular, this factor is independent of $s$. For all finite primes $\pp$, the $\pp$th term in \eqref{eq:gjsdef} has the form $1+A_\pp(s)$, where
\begin{equation}\label{eq:apsexpansion}
A_\pp(s) =  \sum_{\substack{m_\pp \mid m\\ m_\pp > 1}}\frac{\sum_{a_\pp \in \farey(m,m_\pp)} \zeta_j^{a_\pp}}{\abs{\pp}^{n^2(1-1/m_\pp)s}}.
\end{equation}
Since $m$ divides $n$ and $\ell$ is the least prime divisor of $n$, we have $n^2(1-1/m_\pp) \geq n^2(1-1/\ell)$ for each term in the sum. Moreover, each numerator on the right-hand side of \eqref{eq:apsexpansion} is trivially bounded by $m$. It follows that the the formal Dirichlet series expansion of $G_j(s)$ converges absolutely for $\Re(s) > \frac{1}{n^2(1-1/\ell)}$ and coincides there with its (absolutely convergent) Euler product. This gives condition (ii). In fact, if $\ell \nmid m$, then the smallest nontrivial divisor of $m$ is strictly larger than $\ell$. The argument of the preceding paragraph then implies that $G_j(s)$ has an Euler product converging absolutely and uniformly in $\Re(s) > \frac{1}{n^2(1-1/\ell)}-\epsilon$, for some positive $\epsilon$. This shows that conditions (iii) and (iv) hold for $G_j(s)$, where in (iv) we may take $A(s)=0$ and $B(s) = G_j(s)$.

In the case when $\ell \mid m$, we have to analyze the $A_\pp(s)$ more closely. For each $m_\pp$ dividing $m$, the corresponding numerator on the right-hand side of \eqref{eq:apsexpansion} coincides with the Ramanujan sum $c_{m_\pp}(j)$. From H\"{o}lder's explicit evaluation of such sums \cite[Thm 272, p.~309]{HW08},
\begin{align*} \sum_{a_\pp \in \farey(m,m_\pp)} \zeta_j^{a_\pp} &= c_{m_\pp}(j) = \mu\left(\frac{m_\pp}{(m_\pp,j)}\right) \frac{\varphi(m_\pp)}{\varphi(m_\pp/(m_\pp,j))}.  \end{align*}
Thus, the first term on the right-hand sum in \eqref{eq:apsexpansion} --- corresponding to $m_\pp=\ell$ ---  is
\[ \mu\left(\frac{\ell}{(\ell,j)}\right) \frac{\ell-1}{\varphi(\ell/(\ell,j))} \frac{1}{\abs{\pp}^{n^2(1-1/\ell)s}}. \]
For all of the remaining terms, $n^2(1-1/m_\pp) > n^2(1-1/\ell)$. Now if $\ell\nmid j$,
\[ \mu\left(\frac{\ell}{(\ell,j)}\right) \frac{\ell-1}{\varphi(\ell/(\ell,j))} \frac{1}{\abs{\pp}^{n^2(1-1/\ell)s}} = -\frac{1}{\abs{\pp}^{n^2(1-1/\ell)s}}. \]
Put $H_j(s) = G_j(s) \zeta_k(n^2(1-1/\ell) s)$. Since
\[ \zeta_k(n^2(1-1/\ell) s) = \prod_{\pp\text{ finite}} \left(1 + \frac{1}{\abs{\pp}^{n^2(1-1/\ell)s}} + \frac{1}{\abs{\pp}^{2n^2(1-1/\ell)s}} + \dots\right), \]
the $\pp$th factor in the Euler product expansion of $H_j(s)$ has the form $1+O(|\pp|^{-Ns})$ for a positive integer $N$ strictly larger than $n^2(1-1/\ell)$. Thus, $H_j(s)$ continues analytically to $\Re(s) > 1/N$, and so also to the region $\Re(s) \geq \frac{1}{n^2(1-1/\ell)}$. Since $\zeta_k(s)$ has no zeros on $\Re(s)=1$, this gives an analytic continuation of $G_j(s) = H_j(s)\zeta_k(n^2(1-1/\ell)s)^{-1}$ to $\Re(s) \geq \frac{1}{n^2(1-1/\ell)}$. This proves (iii) and (iv) with $A(s)=0$ and $B(s)=G_j(s)$. Now suppose that $\ell \mid j$. Then
\[ \mu\left(\frac{\ell}{(\ell,j)}\right) \frac{\ell-1}{\varphi(\ell/(\ell,j))} \frac{1}{|\pp|^{n^2(1-1/\ell)s}} =\frac{\ell-1}{|\pp|^{n^2(1-1/\ell)s}}. \]
Now arguing with Euler products as above, we find that if we set $H_j(s) = G_j(s) \zeta_k(n^2(1-1/\ell)s)^{-(\ell-1)}$, then $H_j(s)$ is analytic for $\Re(s) \geq \frac{1}{n^2(1-1/\ell)}$. This implies that $G_j(s) = H_j(s) \zeta_k(n^2(1-1/\ell)s)^{\ell-1}$ continues analytically to the same closed half-plane, except for a pole of order at most $\ell-1$ at $s=\frac{1}{n^2(1-1/\ell)}$. Consequently, (iii) and (iv) hold with
\[ A(s) = G_j(s) \pr{s-\frac{1}{n^2(1-1/\ell)}}^{\ell-1} \]
and $B(s)= 0$. Collecting everything, we see that (i)--(iv) all hold for $G(s)$, for $\rho$ and $\beta$ as in \eqref{eq:rhobetachoice}. Moreover, we can take the $A(s)$ in (iv) corresponding to $G(s)$ as $\frac{1}{m}$ times the sum of the functions $A(s)$ corresponding to each $G_j(s)$.
\end{proof}

We now establish the lemma with a few applications of Theorem \ref{DelangeTheorem}. We split into two cases.

\textbf{Case 1.} $\ell \nmid m$. In the case when $\ell \nmid m$, the $A(s)$ corresponding to each $G_j(s)$ was identically zero, hence our final $A(s)$ is also $0$. Thus, Theorem \ref{DelangeTheorem} yields
\[ N_{m,n}(x) = o(x^{\frac{1}{n^2(1-1/\ell)}} (\log{x})^{\ell-2}) \quad \text{ as } x\to\infty. \]
This completes the proof of the lemma in the case $\ell \nmid m$.

\textbf{Case 2.} $\ell \mid m$. If $\ell \mid m$, our work shows that
\[ A(s) = \left(s-\frac{1}{n^2(1-1/\ell)}\right)^{\ell-1} \cdot \frac{1}{m}\pr{\sum_{\substack{0 \leq j < m\\ \ell \mid j}} G_j(s)}. \]
To evaluate this $A(s)$ at $s=\frac{1}{n^2(1-1/\ell)}$, we recall that $\kappa$ denotes the residue at $s=1$ of the simple pole of $\zeta_k(s)$. Writing
\begin{equation*}
\left(s-\frac{1}{n^2(1-1/\ell)}\right)^{\ell-1} G_j(s) = \left(\zeta_k(n^2(1-1/\ell) s) \left(s-\frac{1}{n^2(1-1/\ell)}\right)\right)^{\ell-1} \cdot G_j(s) \zeta_k(n^2(1-1/\ell) s)^{-(\ell-1)},
\end{equation*}
we see that
\begin{equation*}
A\left(\frac{1}{n^2(1-1/\ell)}\right) = \left(\frac{\kappa}{n^2(1-1/\ell)}\right)^{\ell-1} \cdot  \frac{1}{m} \brac{\sum_{\substack{0 \leq j < m \\ \ell \mid j}}\pr{ \lim_{s\to\frac{1}{n^2(1-1/\ell)}} G_j(s) \zeta_k(n^2(1-1/\ell)s)^{-(\ell-1)}}}.
\end{equation*}
It remains to evaluate the limits inside the final sum. For the values of $j$ and $m$ under consideration, $\ell \mid j$ and $\ell \mid m$. So for each finite prime $\pp$, the $\pp$th term in the product expansion \eqref{eq:gjsdef} of $G_j(s)$ begins as
\[ 1+ (\ell-1)/|\pp|^{n^2(1-1/\ell)s} + \dots. \]
Now consider the factors corresponding to infinite primes. If $m$ is odd, then there are no such factors in \eqref{eq:gjsdef}. If $m$ is even, then we must have $\ell=2$, and since $\ell \mid j$,
\[ 1 + \zeta_j^{m/2} = 1 + e^{\pi i j} = 1 + (-1)^j = 2; \]
thus, the factor in \eqref{eq:gjsdef} giving the contribution of the infinite primes is precisely $2^{r_1}$. We conclude that if $m$ is odd, then
\begin{equation*}
\lim_{s\to\frac{1}{n^2(1-1/\ell)}} G_j(s) \zeta_k(s)^{-(\ell-1)} = \prod_{\pp \text{ finite}} \left(1 + \frac{\ell-1}{\abs{\pp}}+ \sum_{\substack{m_\pp \mid m\\ m_\pp > \ell}}\frac{\mu\left(\frac{m_\pp}{(m_\pp,j)}\right) \frac{\varphi(m_\pp)}{\varphi(m_\pp/(m_\pp,j))}}{\abs{\pp}^{(1-1/m_\pp)/(1-1/\ell)}}\right)\left(1-\frac{1}{\abs{\pp}}\right)^{\ell-1},
\end{equation*}
while if $m$ is even, this must be multiplied by $2^{r_1}$. So if $m$ is odd, then
\begin{multline*}
A\left(\frac{1}{n^2(1-1/\ell)}\right) = \left(\frac{\kappa}{n^2(1-1/\ell)}\right)^{\ell-1} \cdot \frac{1}{m} \cdot \\ \sum_{\substack{0 \leq j < m \\ \ell \mid j}} \left(\prod_{\pp \text{ finite}} \left(1 + \frac{\ell-1}{\abs{\pp}}+ \sum_{\substack{m_\pp \mid m\\ m_\pp > \ell}}\frac{\mu\left(\frac{m_\pp}{(m_\pp,j)}\right) \frac{\varphi(m_\pp)}{\varphi(m_\pp/(m_\pp,j))}}{|\pp|^{(1-1/m_\pp)/(1-1/\ell)}}\right)\left(1-\frac{1}{\abs{\pp}}\right)^{\ell-1}\right),\end{multline*}
while if $m$ is even, this expression should be multiplied by $2^{r_1}$.  According to Theorem \ref{DelangeTheorem}, we have
\begin{equation}\label{Bob}
N_{m,n}(x) = \pr{\frac{A\pr{\frac{1}{n^2(1-1/\ell)}}}{\frac{1}{n^2(1-1/\ell)} \Gamma(\ell-1)} +o(1)} x^{\frac{1}{n^2(1-1/\ell)}} (\log{x})^{\ell-2} \quad \text{ as } x\to\infty.
\end{equation}
Comparing Equation (\ref{Bob}) with the definition of $\delta_{m,n}$ in the statement of the lemma, we see the proof is complete.
\end{proof}

We now prove Theorem \ref{thm:area1} from the introduction.

\begin{proof}[Proof of Theorem \ref{thm:area1}]
We view $N(x)$ as counting central simple algebras of the form $\MM(r,D)$ where $r=1$. To single these out, we make use of the well-known identity $\sum_{d \mid r} \mu(d) = 1$ if $r=1$ and $0$ otherwise. Writing $\sum_{A}$ for a sum on central simple algebras $A$ of dimension $n^2$, and $\sum_{A}^{(r)}$ for such a sum restricted to $A$ of the form $\MM(r,D)$, we find that
\[ N(x) = \sum\nolimits_{A}^{(1)}1 = \sum_{m \mid n} \mu(m) \sum_{\substack{r \mid n \\ m \mid r}}\sum\nolimits_{A}^{(r)} 1. \]
Writing $r^2 \dim(D) = n^2$, we see that $m$ divides $r$ if and only if $\dim(D) = d^2$ for divisor $d$ of $n/m$. Hence,
\[ \sum_{\substack{r \mid n \\ m \mid r}}\sum\nolimits_{A}^{(r)} 1 = N_{n/m,n}(x). \]
Putting the last two displays together, we find that
\[ N(x) = \sum_{m \mid n} \mu(m) N_{n/m,n}(x). \]
Replacing $m$ with $n/m$ gives the first statement in the lemma. The asymptotic formula \eqref{eq:Nxasymptotic} with
\[ \delta_n = \sum_{m \mid n} \mu(n/m) \delta_{m,n} \]
now follows from Lemma \ref{lem:keylem0}.

It remains to show that $\delta_n > 0$. Consider first the case when $n=\ell$ is prime. In that case, $\delta_n = \delta_{\ell,n}-\delta_{1,n} = \delta_{\ell,n}$. We used here that $\delta_{1,n}=0$ since $\ell \nmid 1$. From Lemma \ref{lem:keylem0}, $\delta_{\ell,n}$ is given by a product of nonzero factors together with
\[ \prod_{\pp\text{ finite}}\left(1+\frac{\ell-1}{\abs{\pp}}\right) \left(1-\frac{1}{\abs{\pp}}\right)^{\ell-1}. \]
This final product is absolutely convergent and contains only nonzero terms, and so also represents a nonzero real number. This settles the case when $n=\ell$.

Now we treat the case of general $n$. To prove that $\delta_n > 0$, it is enough to construct $\gg x^{\frac{1}{n^2(1-1/\ell)}} (\log{x})^{\ell-2}$ division algebras $A'/k$ of dimension $n^2$ with $\abs{\disc(A')} \le x$. The following crude argument suffices for this purpose. Fix (arbitrarily) finite primes $\pp_1,\dots,\pp_n$ of $k$. We first count division algebras $A/k$ of dimension $\ell^2$ which are unramified at any of $\pp_1, \dots, \pp_n$ and which satisfy $|\disc(A)| \leq X$. Without the ramification condition, we have just seen (the case $n=\ell$) that the number of these $A$ is
\begin{equation}\label{eq:positivedelta}
\gg X^{\frac{1}{\ell^2(1-1/\ell)}} (\log{X})^{\ell-2}
\end{equation}
for large $X$. An entirely analogous proof --- omitting the factors corresponding to $\pp=\pp_1, \dots, \pp_n$ from the Euler products appearing previously --- shows that this lower bound continues to hold with the restrictions on ramification imposed. Now for each such $A$, there is an associated $n^2$--dimensional division algebra $A'/k$ determined by enlarging the set of ramified primes to include $\pp_1, \dots, \pp_n$ and correspondingly enlarging the collection of Hasse invariants to include the numbers $a_{\pp_i}/m_{\pp_i} = 1/n$. (Note that the sum of the numbers $a_{\pp}/m_{\pp}$ for $\pp \mid \Ram(A')$ is one more than the corresponding sum over $\pp \mid \Ram(A)$, so is still an integer.) Clearly, distinct $A$ correspond to distinct $A'$. Moreover,
\[ \abs{\disc(A')} = \abs{\pp_1 \cdots \pp_n}^{n^2(1-1/n)} \abs{\disc(A)}^{(n/\ell)^2}. \]
Thus, $|\disc(A')|\le x$ precisely when $|\disc(A)|\le X$, where \[ X:= (x/|\pp_1 \cdots \pp_n|^{n^2(1-1/n)})^{(\ell/n)^2}. \]
Plugging this value of $X$ into \eqref{eq:positivedelta}, we see that we have constructed $\gg x^{\frac{1}{n^2(1-1/\ell)}} (\log{x})^{\ell-2}$ division algebras $A'/k$ of dimension $n^2$ with $|\disc(A')|\le x$, for all large $x$.
\end{proof}

\begin{ex}
The explicit expression for $\delta_n$ is, in general, exceedingly complicated. However, it can be written fairly compactly in certain special cases. To begin with, suppose that the smallest prime factor $\ell$ of $n$ is odd. If $n=\ell$, then $\delta = \delta_{\ell,n}$, where
\begin{equation}\label{eq:deltaelln}
\delta_{\ell,n} = \frac{\kappa^{\ell-1}}{\ell(\ell-2)!} \frac{1}{(n^2(1-1/\ell))^{\ell-2}} \prod_{\pp\text{ finite}}\left(1+\frac{\ell-1}{|\pp|}\right) \left(1-\frac{1}{|\pp|}\right)^{\ell-1}.
\end{equation}
Next, suppose that $n=\ell^2$. Then $\delta_n = \delta_{\ell^2,\ell^2}-\delta_{\ell,\ell^2}$. The second term can be calculated with \eqref{eq:deltaelln}, while
\begin{multline*}
\delta_{\ell^2,\ell^2} = \frac{\kappa^{\ell-1}}{\ell^2(\ell-2)!} \cdot \frac{1}{(\ell^4 (1-1/\ell))^{\ell-2}} \cdot \\ \left(\prod_{\pp\text{ finite}}\bigg(1 + \frac{\ell-1}{\abs{\pp}} +\frac{\ell(\ell-1)}{\abs{\pp}^{1+1/\ell}}\right)\left(1-\frac{1}{\abs{\pp}}\right)^{\ell-1}+ \\(\ell-1)\prod_{\pp\text{ finite}}\left(1 + \frac{\ell-1}{\abs{\pp}} -\frac{\ell}{\abs{\pp}^{1+1/\ell}}\right)\left(1-\frac{1}{\abs{\pp}}\right)^{\ell-1}\bigg).
\end{multline*}
Finally, suppose that $n=\ell \ell'$, where $\ell'$ is a prime larger than $\ell$. Then $\delta_n = \delta_{\ell \ell',\ell \ell'} - \delta_{\ell, \ell \ell'} - \delta_{\ell', \ell \ell'}$, where the second and third terms can be computed with \eqref{eq:deltaelln}, and
\begin{multline*}
\delta_{\ell \ell'} = \frac{\kappa^{\ell-1}}{\ell \ell' (\ell-2)!} \cdot \frac{1}{(\ell^2 \ell'^2 (1-1/\ell))^{\ell-2}} \cdot
\\ \bigg(\prod_{\pp\text{ finite}}\left(1+\frac{\ell-1}{\abs{\pp}}+\frac{\ell'-1}{|\pp|^{\frac{1-1/\ell'}{1-1/\ell}}}+ \frac{\ell(\ell-1)}{\abs{\pp}^{1+1/\ell}}\right)\left(1-\frac{1}{\abs{\pp}}\right)^{\ell-1} + \\
(\ell'-1) \prod_{\pp\text{ finite}}\left(1+\frac{\ell-1}{\abs{\pp}}-\frac{1}{\abs{\pp}^{\frac{1-1/\ell'}{1-1/\ell}}}- \frac{\ell'-1}{\abs{\pp}^{1+1/\ell}}\right)\left(1-\frac{1}{\abs{\pp}}\right)^{\ell-1}
\bigg).
\end{multline*}
If $\ell=2$, the same analysis applies, but all of these expressions for $\delta_n$ must be multiplied by $2^{r_1}$.
\end{ex}

\subsection{Proof of Theorem \ref{thm:countingquads}}

In this section, we fix a number field $k$ and a quaternion algebra $B$ defined over $k$ and count the number of quadratic extensions $L/k$ with norm of relative discriminant less than $X$ which embed into $B$. In what follows, denote by $\Delta_{L/k}$ the relative discriminant of $L$ over $k$. If $P$ is any property a quadratic extension of $k$ may have, we make the definition
\begin{equation}\label{eq:probdef}
\Prob(P):= \lim_{x\to\infty} \frac{\#\{\text{quadratic extensions $L/k$ for which $P$ holds and $\abs{\Delta_{L/k}}\leq x$}\}}{\#\{\text{quadratic extensions $L/k$ with $\abs{\Delta_{L/k}}\leq x$}\}},
\end{equation}
provided that this limit exists. The next result, which is a special case of results of Wood \cite{wood10}, asserts that for certain properties $P$ related to splitting behavior, these ``probabilities'' behave as one might naively expect.

\begin{prop}[Wood]\label{prop:wood}
Fix a finite collection $S$ of real or finite places of $k$. For each $\pp \in S$, let $P_\pp$ be one of the properties ``$\pp$ ramifies in $L$'', ``$\pp$ splits in $L$'', or ``$\pp$ is inert in $L$'', subject to the restriction that $P_\pp$ is one of the first two if $\pp$ is a real place. Then:
\begin{enumerate}
\item $\Prob(P_\pp)$ exists for each $\pp \in S$.
\item $\displaystyle\Prob(\text{all $P_\pp$ hold at once}) = \prod_{\pp \in S} \Prob(P_\pp)$.
\item If $\pp$ is real, then $\displaystyle\Prob(\pp\text{ ramifies}) = \Prob(\pp\text{ splits})= \frac{1}{2}$.
\item If $\pp$ is finite, then $\Prob(\pp\text{ splits})= \Prob(\pp\text{ is inert}) = \frac{1}{2} (1-\Prob(\pp\text{ ramifies}))$.
\end{enumerate}
\end{prop}

It is worth saying a few words about how Proposition \ref{prop:wood} follows from the more general results of Wood. In Wood's terminology, we are counting $\Z/2\Z$--extensions of $k$ with local specifications. We note however that Wood's definition of a local specification differs from our simplified picture above, but only in the sense that it is strictly finer; she allows one to specify the $k$--algebra $L\otimes_{k} k_\pp$, which gives us more information than we are measuring. When $G=\Z/2\Z$, all local specifications are viable (see \cite[start of \S2.2]{wood10}), and counting by discriminant defines a fair counting function (in the sense of \cite[\S2.1]{wood10}). The existence of each $\Prob(P_\pp)$ in Proposition \ref{prop:wood} can be seen as a special case of \cite[Thm 2.1]{wood10}. The independence result follows from \cite[Cor 2.4]{wood10}. The statement about the splitting behavior of real primes comes from \cite[Cor 2.2]{wood10}, while the statement about the behavior of finite primes is guaranteed by \cite[Cor 2.3]{wood10}.

Theorem \ref{thm:countingquads} follows easily from Proposition \ref{prop:wood} and the following estimate of Datskovsky and Wright \cite{DW88} for the denominator appearing in the definition \eqref{eq:probdef} (see \cite[\S2.2]{cohen02} for an alternative proof of this proposition).

\begin{prop}\label{prop:cohen}
The number of quadratic extensions $L/k$ with $\abs{\Delta_{L/k}} \leq x$ is $\sim \frac{1}{2^{r_2}} \frac{\kappa_k}{\zeta_k(2)} x$ as $x\to\infty$, where $\kappa_k$ denotes the residue at $s=1$ of $\zeta_k(s)$ and $r_2$ is the number of pairs of complex embeddings of $k$.
\end{prop}

\begin{proof}[Deduction of Theorem \ref{thm:countingquads}]
Recall that $L$ embeds into $B$ precisely when every prime dividing the discriminant of $B$ is non-split in $L$. The probability that a fixed real prime of $k$ ramifies in $L$ is $\frac{1}{2}$ (from Proposition \ref{prop:wood}(iii)), while the probability that a fixed finite prime of $k$ is inert or ramified in $L$ is (from Proposition \ref{prop:wood}(iv))
\[ \frac{1}{2}\pr{1+\Prob(\pp\text{ ramifies})} \geq \frac{1}{2}. \]
So from Proposition \ref{prop:wood}(ii),  the probability that $L$ embeds into $B$ exists and is at least $\frac{1}{2^{r'}}$, where $r'$ is the number of distinct places dividing the discriminant of $B$. Theorem \ref{thm:countingquads} now follows from the estimate of Proposition \ref{prop:cohen}.
\end{proof}

\subsection{Proof of Theorem \ref{thm:quatembedding}}

For a number field $k$ and a quadratic extension $L/k$, our present goal is to count the number of quaternion algebras over $k$ which have discriminant less than $x$ and which admit an embedding of $L$. In fact, we solve a more general problem. Specifically, in this subsection we prove Theorem \ref{thm:quatembedding} from the introduction.

\subsubsection{Setup}

From Theorem \ref{lem:allDAs}, a quaternion algebra $B/k$ is uniquely specified by a finite set $S \subset \Pp_k$ of real and finite places of $k$, along with a reduced fraction $0 < a_{\pp}/{m_{\pp}} < 1$ for each prime $\pp \in S$, where $\lcm_{\pp \in S}[m_{\pp}]=2$ and $\sum_{\pp \in S} a_{\pp}/{m_{\pp}} \in \Z$. The least common multiple condition forces each ${a_{\pp}}/{m_{\pp}} = 1/2$, and the integrality condition on the sum forces the cardinality of $S$ to be even.  We conclude that there is a bijection between quaternion algebras over $k$ and square-free moduli $\m$ of $k$ containing a nonzero even number of factors. Moreover, if $B$ corresponds to $\m$ under this bijection, then
\[ \disc(B) = \bigg(\prod_{\substack{\pp \text{ finite}\\ \pp \mid \m}} \pp\bigg)^2 \cdot \prod_{\substack{\pp \text{ real}\\ \pp \mid \m}} \pp. \]

Now let $\Qq$ be the set of finite or real primes of $k$ that do not split in any of the $L_i$.  Asking that all of our quadratic extensions $L_i$ embed into the quaternion algebra $B/k$ amounts to requiring that $\m$ only be divisible by primes residing in $\Qq$. We count the number of such $B$ with $\abs{\disc(B)} \le x$ by modifying the approach of the last section. We now provide the details. Define $G(s) = \frac{1}{2} (G_0(s) + G_1(s))$, where
\begin{equation}\label{eq:gdef}
G_0(s) = \prod_{\substack{\pp \text{ real}\\ \pp \in \Qq}}\left(1+\frac{1}{\abs{\pp}^s}\right) \prod_{\substack{\pp \text{ finite}\\ \pp \in \Qq}}\left(1+\frac{1}{\abs{\pp}^{2s}}\right)
\end{equation}
and
\begin{equation}\label{eq:gdef2}
G_1(s) = \prod_{\substack{\pp \text{ real}\\ \pp \in \Qq}}\left(1-\frac{1}{\abs{\pp}^s}\right) \prod_{\substack{\pp \text{ finite}\\ \pp \in \Qq}}\left(1-\frac{1}{\abs{\pp}^{2s}}\right).
\end{equation}
The infinite factors in the definitions of $G_0(s)$ and $G_{1}(s)$ are in fact independent of $s$. If $r_1'$ is the number of real primes of $k$ that do not split in any of the $L_i$, the contribution of the infinite factors is given by $2^{r_1'}$ and $0^{r_1'}$, respectively, where
\[ 0^{r_1'} = \begin{cases} 1, & r_1' = 0, \\ 0, & \text{otherwise}. \end{cases} \]
Now observe that if $G(s)$ is identified with its formal Dirichlet series expansion, then the coefficient of $N^{-s}$ counts  quaternion algebras $B/k$ with $\abs{\disc(B)}=N$ admitting an embedding of all $L_i$.

To estimate the partial sums of the $G(s)$--coefficients for $N\leq x$, we work with the corresponding sums for $G_0(s)$ and $G_1(s)$ individually. For $G_0(s)$, we will apply Theorem \ref{DelangeTheorem} to obtain an asymptotic formula with main term proportional to $x^{1/2}/(\log{x})^{1-\frac{1}{2^r}}$. We then use a result of Wirsing \cite[Satz 2]{wirsing} to show that the partial sums of the $G_1$--coefficients are in fact $o(x^{1/2}/(\log{x})^{1-\frac{1}{2^r}})$ as $x\to\infty$. Putting these estimates together yields Theorem \ref{thm:quatembedding}.

\subsubsection{The partial sums of the coefficients of $G_0(s)$}

Let us check that the hypotheses of Theorem \ref{DelangeTheorem} hold with $\rho = \frac{1}{2}$ and $\beta=\frac{1}{2^r}$. Conditions (i) and (ii) of that theorem are clear from the product definition of $G_0$, and so we may focus on (iii) and (iv). Let $L$ be the composite field of the $L_i$, for $i=1,2,\dots, r$. The essential idea is to express $G_0(s)$ in terms of benign factors and Artin $L$--functions attached to $\Gal(L/k)$. We now implement this idea. Our assumption that $[L_1 \cdots L_r: k]=2^r$ easily implies that $\Gal(L/k)$ is canonically isomorphic to the elementary abelian $2$--group
\[ \bigoplus_{i=1}^{r} \Gal(L_i/k) = \bigoplus_{i=1}^r \Z/2\Z = \pr{\Z/2\Z}^r. \]
For $1 \leq i \leq r$, let $\tilde{\chi}_i$ denote the unique character of $\Gal(L/k)$ whose kernel is the subgroup fixing $L_i$. For every subset $T$ of $\{1, \dots, r\}$, let $\tilde{\chi}_T = \prod_{i \in T}\tilde{\chi}_i$. It is a simple matter that $\tilde{\chi}_{\Ss}$ is nontrivial provided $T$ is non-empty. Consequently, the field $L_T$ left fixed by $\ker\tilde{\chi}_{T}$ is a quadratic extension of $k$ in the event $T$ is non-empty. For each finite prime $\pp$ of $k$, we let $\chi_{T}(\pp)=1, 0$, or $-1$ according to whether $\pp$ splits, ramifies, or remains inert in $L_T$. When $T$ consists of a single element $1 \leq i \leq r$, we will write $\chi_i$ instead of the more cumbersome $\chi_{\{i\}}$. In that notation, unless $\pp$ belongs to
\[ \Rr:=\{\text{finite primes of $k$ that ramify in $L$}\}, \]
we have the expression
\[ \chi_T(\pp) = \tilde{\chi}_{T}(\Frob(\pp)) = \prod_{i \in T}\tilde{\chi_i}(\Frob(\pp)) = \prod_{i \in T}\chi_i(\pp),\]
where $\Frob(\pp) \in \Gal(L/k)$ is the associated Frobenius automorphism for $\pp$.

Now we relate the $\chi_i$ to the definition of $G_0$. If $\pp$ is finite and not an element of $\Rr$, then
\[ \frac{1}{2^r} \prod_{i=1}^{r} (1-\chi_i(\pp)) = 1 \]
if $\pp \in \Qq$ and $=0$ otherwise. Hence, setting $\Qq_0:= \Qq \cap \Rr$, we have the expression
\begin{equation}\label{eq:g0expr}
G_0(s) = 2^{r_1'} \prod_{\pp \in \Qq_0} \left(1+\frac{1}{\abs{\pp}^{2s}}\right) \prod_{\substack{\pp\text{ finite}\\\pp \notin \Rr}} \left(1+ \frac{\frac{1}{2^r} \prod_{i=1}^{r}(1-\chi_i(\pp))}{|\pp|^{2s}}\right).
\end{equation}

For $\pp \notin \Rr$, we have
\begin{equation}\label{eq:givescancelation} \prod_{i=1}^{r}(1-\chi_i(\pp)) = \sum_{T \subset \{1, 2, \dots, r\}} (-1)^{\# T} \chi_T(\pp). \end{equation}
For $\Re(s) > 1/2$, let
\[ Z_0(s):= \prod_{\substack{\pp\text{ finite}\\\pp \notin \Rr}} \left(\left(1+ \frac{\frac{1}{2^r} \prod_{i=1}^{r}(1-\chi_i(\pp))}{\abs{\pp}^{2s}}\right)^{2^r} \prod_{T \subset \{1, 2, \dots, r\}}  \left(1-\frac{\chi_T(\pp)}{|\pp|^{2s}}\right)^{(-1)^{\# T}}\right). \]
Recalling that $\log(1+t) = \sum_{j \ge 1} (-1)^{j-1} t^{j}/j$ for $|t| < 1$, we see that
\begin{align*} \log Z_0(s) &= \sum_{\substack{\pp\text{ finite}\\\pp \notin \Rr}} \sum_{j \ge 1} \left(\frac{(-1)^{j-1}}{j} 2^r \left(\frac{1}{2^r} \prod_{i=1}^{r}(1-\chi_i(\pp))\right)^j -\sum_{T \subset \{1, 2, \dots, r\}}\frac{(-1)^{\#T} \chi_T(\pp)^{j}}{j}\right)\abs{\pp}^{-2js} \\
&= \sum_{\substack{\pp, j\\\pp\text{ finite},~\pp \notin \Rr \\ j\ge 2}} \left(\frac{(-1)^{j-1}}{j} 2^r \left(\frac{1}{2^r} \prod_{i=1}^{r}(1-\chi_i(\pp))\right)^j -\sum_{T \subset \{1, 2, \dots, r\}}\frac{(-1)^{\#T} \chi_T(\pp)^{j}}{j}\right)\abs{\pp}^{-2js}. \end{align*}
Here we used \eqref{eq:givescancelation} to pass from the first line to the second. The final  sum is absolutely convergent for $\Re(s) > \frac14$, and any arrangement of the sum converges uniformly on compact subsets of $\Re(s) > \frac{1}{4}$. Thus, $\log Z_0(s)$ continues analytically to $\Re(s) > \frac14$, implying that $Z_0(s)$ can be extended to a function that is analytic and nonzero there.

For each $T \subset \{1, 2, \dots, r\}$ and all $s$ with real part greater than $1$, put
\[ L(s,\chi_{T}) = \prod_{\pp\text{ finite}} \left(1-\frac{\chi_{T}(\pp)}{\abs{\pp}^s}\right)^{-1}. \]
$L(s,\chi_{T})$ is the Artin $L$--function attached to the character $\tilde{\chi}_T$ of $\Gal(L/k)$. When $T=\emptyset$, we have $L(s,\chi_{T})= \zeta_k(s)$, and for all other choices of $T$, the function $L(s,\chi_{T})$ is analytic and nonzero for $\Re(s) \geq 1$.

We chose $Z_0(s)$ so that
\[
\prod_{\substack{\pp\text{ finite}\\\pp \notin \Rr}} \left(1+ \frac{\frac{1}{2^r} \prod_{i=1}^{r}(1-\chi_i(\pp))}{\abs{\pp}^{2s}}\right)^{2^r} = Z_0(s) \prod_{T \subset \{1, 2, \dots, r\}} \left(L(2s,\chi_{T}) \prod_{\pp \in \Rr} \left(1-\frac{\chi_T(\pp)}{\abs{\pp}^{2s}}\right)\right)^{(-1)^{\# T}}.
\]
The right-hand products over $\pp \in \Rr$ are analytic and nonzero for $\Re(s) > 0$. Keeping in mind that $L(s,\chi_T)$ is analytic and nonzero for $\Re(s) \ge 1$ as long as $T\neq \emptyset$, we see that the expression
\begin{equation}\label{eq:longexpression}
\pr{s-\frac{1}{2}} Z_0(s) \prod_{T \subset \{1, 2, \dots, r\}} \pr{L(2s,\chi_{T}) \prod_{\pp \in \Rr} \pr{1-\frac{\chi_T(\pp)}{\abs{\pp}^{2s}}}}^{(-1)^{\# T}}
\end{equation}
is analytic and nonzero for $\Re(s) \geq \frac{1}{2}$. We note for the reader that the factor of $s-\frac{1}{2}$ here is used to cancel the simple pole of $\zeta_k(2s)$ at $s=\frac{1}{2}$. A function which is nonzero and analytic on a simply connected domain possesses an analytic logarithm on the same domain, and hence also an analytic $N$th root for every $N$. (See, e.g., \cite[Thm 13.11, p.~274]{Rudin}.) In particular, we can extract a $2^r$th root $H_0(s)$ of \eqref{eq:longexpression} which is also analytic and nonzero in $\Re(s) \geq \frac{1}{2}$. The choice of $H_0$ is uniquely specified if we insist that $H_0(s) > 0$ for $s > \frac{1}{2}$. Referring back to \eqref{eq:g0expr}, we find that for $\Re(s) > \frac{1}{2}$,
\[ G_0(s) = \frac{1}{(s-\frac{1}{2})^{1/2^r}} \left(2^{r_1'} \cdot \prod_{\pp \in \Qq_0}\left(1+\frac{1}{\abs{\pp}^{2s}} \right) \cdot H_0(s)\right), \]
where $(s-\frac{1}{2})^{1/2^r}$ is the principal $2^r$th root. This immediately implies (iii). If we let
\[ A_0(s):= 2^{r_1'} \cdot \prod_{\pp \in \Qq_0}\left(1+\frac{1}{\abs{\pp}^{2s}} \right) \cdot H_0(s), \]
we see that $G_0(s)$ satisfies (iv) with $\rho=\frac{1}{2}$, $\beta=\frac{1}{2^r}$, $A(s) = A_0(s)$, and $B(s)=0$.

So by Theorem \ref{DelangeTheorem}, the sum up to $x$ of the coefficients of $G_0(s)$ is asymptotic to
\[ \frac{A_0\pr{\frac{1}{2}}}{\frac{1}{2} \Gamma\pr{\frac{1}{2^r}}} x^{\frac{1}{2}}/(\log{x})^{1-\frac{1}{2^r}} \quad \text{ as } x \to\infty. \]

\subsubsection{The coefficients of $G_1(s)$}

We now show that the contribution from the coefficients of $G_1(s)$ is negligible. Define arithmetic functions $a(N)$ and $b(N)$ by expanding
\[ \prod_{\substack{\pp\text{ finite} \\ \pp \in \Qq}} \left(1+\frac{1}{\abs{\pp}^s}\right) = \sum_{N=1}^{\infty}\frac{a(N)}{N^s} \quad\text{and}\quad \prod_{\substack{\pp\text{ finite} \\ \pp \in \Qq}} \left(1-\frac{1}{\abs{\pp}^s}\right) = \sum_{n=1}^{\infty}\frac{b(N)}{N^s}. \]
The functions $a(N)$ and $b(N)$ are multiplicative and satisfy $|b(N)| \leq a(N)$ for all $N$. Referring back to the earlier definitions of $G_0$ and $G_1$, we see that the partial sum of the $G_0(s)$--coefficients up to $x$ is given by
\[ 2^{r_1} \sum_{N \leq \sqrt{x}} a(N), \]
while that of the $G_1$--coefficients is given by
\[ 0^{r_1} \sum_{N \leq \sqrt{x}} b(N). \]
Thus, if we can show that
\begin{equation}\label{eq:smallerthan}
\sum_{N \leq x} b(N) = o\pr{\sum_{N \leq x} a(N)}\qquad \text{as $x\to\infty$},
\end{equation}
then the partial sums of the $G_1$--coefficients are $o(x^{1/2}/(\log{x})^{1-\frac{1}{2^r}})$, as desired. For that task, we use the following result, which is a slight variant of a theorem of Wirsing \cite[Satz 2]{wirsing}.

\begin{prop}[Wirsing]\label{prop:wirsing}
Let $a(N)$ be a multiplicative function taking only nonnegative values. Assume
\begin{enumerate}
\item[(i)] there is a constant $\tau > 0$ for which $\sum_{p \leq x} a(p) \log{p} = (\tau + o(1)) x$ as $x\to\infty$,
\item[(ii)] $a(p^\ell)$ is bounded uniformly on prime powers $p^\ell$ with $\ell\geq 2$.
\end{enumerate}
 Let $b(N)$ be a complex-valued multiplicative function satisfying $\abs{b(N)} \leq a(N)$ for all $N$. Suppose moreover that
\begin{enumerate}
\item[(iii)] there is a constant $\tau'\neq \tau$ with $\sum_{p \leq x} b(p) \log{p} = (\tau'+o(1)) x$ as $x\to\infty$,
\end{enumerate}
then $\displaystyle\sum_{N \leq x} b(N) = o\left( \sum_{N \leq x} a(N)\right)$.
\end{prop}

Here we have replaced the original condition (7) in \cite[Satz 2]{wirsing} with the simpler condition that $\tau \neq \tau'$; this is justified in the remarks immediately following the statement of Satz 2 in \cite{wirsing}. Wirsing's original formulation also assumes asymptotics for the sums $\sum_{p \leq x} f(p)$ rather than $\sum_{p \leq x} f(p) \log p$. The conditions that Wirsing needs follow from our (i), (iii) after applying partial summation. (A similar partial summation argument can be seen in the second displayed equation in the proof of Theorem 4.4 on p. 79 of \cite{apostol}.)


 These remarks allow us to replace condition (7) in \cite[Satz 2]{wirsing} with the simpler condition that $\tau \neq \tau'$. We now return to deducing (\ref{eq:smallerthan}).

\begin{proof}[Proof of \eqref{eq:smallerthan}]
Let us check that the hypotheses of Proposition \ref{prop:wirsing} are satisfied for our choice of $a(N)$ and $b(N)$ above. We have
\[ \sum_{p \leq x} a(p) \log{p} = \sum_{p \leq x} \pr{\sum_{\substack{\pp \in \Qq, \text{ finite} \\ \abs{\pp}=p}} \log\abs{\pp}} = \sum_{\substack{\pp \in \Qq, \text{ finite} \\ \pp\text{ abs. degree $1$}\\ \abs{\pp}\leq x}} \log\abs{\pp} = \sum_{\substack{\pp \in \Qq, \text{ finite} \\ \abs{\pp}\leq x}} \log\abs{\pp} + O(x^{1/2}). \]
Now $\pp \in \Qq$ if and only if $\pp$ does not split in any of the $L_i$. Excluding the finitely many primes in $\Rr$, this condition on $\pp$ is equivalent to the requirement that $\Frob(\pp)$ not restrict to the identity on $L_i$ for any $i=1,2,\dots, r$. Recalling that
\[ \Gal(L/k) \cong \prod_{i=1}^{r}\Gal(L_i/k), \]
we see that this condition uniquely determines $\Frob(\pp)$. Since $\#\Gal(L/k)=2^r$, we can apply partial summation along with the Chebotarev density theorem for natural density (see \cite[Satz 4]{artin}) to obtain
\[ \sum_{\substack{\pp \in \Qq, \text{ finite} \\ \abs{\pp}\leq x}} \log\abs{\pp} = \left(\frac{1}{2^r}+o(1)\right) x \quad \text{ as } x\to\infty. \]
(An analogous partial summation argument appears in the first displayed equation in the proof of Theorem 4.4 on p. 79 of \cite{apostol}; we use the Chebotarev prime counting function from \cite[Satz 4]{artin} in lieu of $\pi(x)$.) So (i) holds with $\tau = \frac{1}{2^r}$.

For each prime power $p^\ell$,
\[ a(p^\ell) \leq \#\set{\text{square-free ideals of $\calO_k$ of norm $p^\ell$}}. \]
However, any square-free ideal of norm $p^\ell$ must be a square-free product of the primes lying above $p$, and there are at most $2^{[k:\Q]}$ such products. This gives (ii).

Finally, our work towards (i) shows that
\[ \sum_{p \leq x} b(p) \log{p} = -\sum_{p \leq x} \sum_{\substack{\pp \in \Qq, \text{ finite} \\ \abs{\pp}=p}} \log\abs{\pp} = -\left(\frac{1}{2^r} + o(1)\right) x \quad \text{ as } x\to\infty. \]
Hence, (iii) holds with $\tau' = -\frac{1}{2^r}$.
\end{proof}

\subsubsection{Denouement}

Since $G(s) = \frac{1}{2}(G_0(s) + G_1(s))$, combining the results of the previous two sections shows that the number of quaternion algebras $B/k$ admitting embeddings of all of the $L_i$ and having $\abs{\disc(B)} \le x$, is asymptotically
\[  \frac{A_0\pr{\frac{1}{2}}}{\Gamma\pr{\frac{1}{2^r}}} \cdot x^{1/2}/(\log{x})^{1-\frac{1}{2^r}}. \]
The leading coefficient here is nonzero and can be given explicitly in terms of the leading terms in the Laurent series expansions for the functions $L(s,\chi_{T})$. Specifically, tracing back through the definitions, we see that with $\kappa$ equal to the residue at $s=1$ of $\zeta_k(s)$,
\begin{multline*}
\frac{A_0\pr{\frac{1}{2}}}{\Gamma\pr{\frac{1}{2^r}}} = \frac{2^{r_1'-\frac{1}{2^r}}}{\Gamma\pr{\frac{1}{2^r}}} \prod_{\pp \in \Qq_0} \left(1+\frac{1}{\abs{\pp}}\right) \cdot \\ \left(\kappa \prod_{\pp \in \Rr}\left(1-\frac{1}{\abs{\pp}}\right) \prod_{\substack{T\subset\{1, 2, \dots, r\} \\ T \neq \emptyset}} \left(L(1,\chi_{T}) \prod_{\pp \in \Rr}\left(1-\frac{\chi(\pp)}{\abs{\pp}}\right)\right)^{(-1)^{\# T}} \right)^{1/2^r} \cdot Z^{1/2^r},
\end{multline*}
where
\begin{align*}
Z :&= Z_0\left(\frac{1}{2}\right) = \prod_{\substack{\pp\text{ finite}\\\pp \notin \Rr}} \left(\left(1+ \frac{\frac{1}{2^r} \prod_{i=1}^{r}(1-\chi_i(\pp))}{\abs{\pp}}\right)^{2^r} \prod_{T \subset \{1, 2, \dots, r\}}  \left(1-\frac{\chi_T(\pp)}{\abs{\pp}}\right)^{(-1)^{\#T}}\right).
\end{align*}
It is clear that in general, the explicit form of this leading coefficient is rather complicated, but in the case $r=1$ it simplifies nicely to the formula given in Example \ref{ex:quatembeddingexample}.

\begin{rmk} Without the assumption that $[L_1 \cdots L_r:k]= 2^r$, it is possible for there to be \emph{no} quaternion division algebras $B/k$ into which all of the $L_i$ embed. For example, take $k=\Q$ and consider the collection of $L_i$ given by $\Q(\sqrt{-3})$, $\Q(\sqrt{-1})$, $\Q(\sqrt{3})$, $\Q(\sqrt{10})$, $\Q(\sqrt{17})$. One can check that every finite prime of $\Q$ splits in at least one of these $L_i$, and so a quaternion algebra into which all of these fields embeds must have a discriminant not divisible by any finite prime at all. As a quaternion algebra must ramify at a finite even number of primes, up to isomorphism, the only quaternion algebra admitting embeddings of all of the above extensions is $\MM(2,\Q)$, which is not a division algebra.
	
For the general situation, we proceed as follows. Suppose we are given a finite collection of distinct quadratic extensions $L_i/k$. Let $L$ be the compositum of all of the $L_i$, and define $r$ by the condition $[L:k] = 2^r$. Note, we are not assuming here that $r$ is the total number of $L_i$. For each $i$, let $\tilde{\chi}_i$ be the character of $\Gal(L/k)$ whose kernel is the subgroup fixing $L_i$. For there to be infinitely many quaternion algebras $B/k$ into which all of the $L_i$ embed, it is necessary and sufficient that there is no finite \textbf{odd order} subset of the $\tilde{\chi}_i$ which multiply to the identity. If this condition holds, then the number of $B/k$ with $\abs{\disc(B)}\le x$ into which all of the $L_i$ embed is again asymptotic to $\delta x^{1/2}/(\log{x})^{1-\frac{1}{2^r}}$ for some positive $\delta > 0$. This can be established by slight modifications to the proof of Theorem \ref{thm:quatembedding}.
\end{rmk}

\section{Main tools: Geometric counting results}\label{Maintools:Geometriccount}

In this section, we derive the geometric counting results from the introduction using the tools from the previous section.

\subsection{Proof of Corollaries \ref{cor:area1cor2} and \ref{cor:area1cor1}}

As an application of Theorem \ref{thm:area1} we consider the problem of counting commensurability classes of arithmetic hyperbolic 2-- and $3$--manifolds with a fixed trace field $k$. As Selberg's lemma ensures every complete, finite volume hyperbolic $n$--orbifold has a finite manifold cover, counting commensurability classes of arithmetic orbifolds is the same as counting commensurability classes of arithmetic manifolds. Consequently, we will not fret about whether our representatives are manifolds or orbifolds. It is well-known \cite[Ch 11]{MR} that given such a commensurability class $\scrC$, there is a real number $V_{\scrC}>0$ which occurs as the smallest volume achieved by an orbifold belonging to this class. A consequence of Borel's classification of maximal arithmetic Fuchsian and Kleinian groups and their volumes \cite{borel-commensurability} is that we can derive a precise formula for $V_{\scrC}$ in terms of the number theoretic invariants of $\scrC$. The proofs of Corollary \ref{cor:area1cor2} and Corollary \ref{cor:area1cor1} will rely crucially on Theorem \ref{CommensurabilityInvariants}. Namely, every commensurability class of arithmetic hyperbolic $2$-- or $3$--manifolds both determines and is determined by the associated trace field and quaternion algebra.

We begin by proving a lemma which bounds the norm of the discriminant of the quaternion algebra of a compact arithmetic hyperbolic $2$-- or $3$--manifold as a function of the volume $V$ of the manifold.

\begin{lem}\label{lem:Bdiscbound}
Let $M$ be a compact arithmetic hyperbolic 2--manifold (resp., $3$--manifold) of volume $V$ with trace field $k$ and quaternion algebra $B$. Then $\abs{\disc(B)}\leq \left[10^{93}V^{13}\right]^{10}$ (resp., $\abs{\disc(B)}\leq 10^{57}V^{7}$).
\end{lem}

\begin{proof}
We establish the lemma for $3$--manifolds as the case of hyperbolic surfaces is similar. Towards that goal, set $V^\prime$ to be the covolume of a minimal covolume maximal arithmetic subgroup in the commensurability class associated to $B$ and $k$. It is known by Chinburg--Friedman \cite[pp. 8]{chinburg-smallestmanifold} that
\begin{equation}\label{equation:minvol1}
V^\prime=\frac{2\pi^2\zeta_k(2)d_k^{\frac{3}{2}}\Phi(\disc(B))}{(4\pi^2)^{n_k}[k_B:k]},
\end{equation}
where
\[ \Phi(\disc(B))=\prod_{\pp \mid \disc(B)}\pr{\frac{\abs{\pp}-1}{2}} \]
and $k_B$ is the maximal abelian extension of $k$ which has $2$--elementary Galois group, is unramified at all finite primes of $k$ and in which all (finite) prime divisors of $\disc(B)$ split completely. As $k_B$ is contained in the strict class field of $k$,
\[ [k_B:k]\leq 2^{r_1(k)}h_k=2^{n_k-2}h_k. \]
Let $\omega_2(B)$ denote the number of prime divisors of $\disc(B)$ which have norm $2$. From the Euler product expansion of $\zeta_k(s)$, we deduce that $\zeta_k(2)\geq (\frac{4}{3})^{\omega_2(B)}$. In combination with (\ref{equation:minvol1}), we conclude that
\begin{equation}\label{Local13}
V^\prime \geq\frac{(\frac{4}{3})^{\omega_2(B)}d_k^{\frac{3}{2}}\abs{\disc(B)}}{(4\pi^2)^{n_k}4^{\omega_2(B)}2^{n_k-2}h_k} \geq \frac{d_k^{\frac{3}{2}}\abs{\disc(B)}}{(8\pi^2)^{n_k}3^{\omega_2(B)}h_k}.
\end{equation}
Now, we have the trivial bound $\omega_2(B)\leq n_k$ and the inequality
$h_k\leq 242d_k^{\frac{3}{4}}$
found in \cite[Lemma 3.1]{linowitz-isospectralsize}. Coupling these two inequalities with (\ref{Local13}) produces
\begin{equation}\label{Local14}
V^\prime \geq \frac{d_k^{\frac{3}{2}}\abs{\disc(B)}}{242(24\pi^2)^{n_k}d_k^{\frac{3}{4}}} \geq\frac{\abs{\disc(B)}}{242(24\pi^2)^{n_k}}.
\end{equation}
Our proof is now complete upon applying \cite[Lemma 4.3]{chinburg-smallestorbifold}, which implies that $n_k\leq 23+\log(V^\prime)$, to (\ref{Local14}) in tandem with the fact that $V\geq V^\prime$.
\end{proof}


\begin{proof}[Proof of Corollary \ref{cor:area1cor2}]
Let $k$ be a totally real number field and $B$ be a quaternion division algebra over $k$ which is ramified at all but one real places of $k$. If $\rho\colon B\to \MM(2,\R)$ is a representation and $\calO$ is an order of $B$, then it is easy to see that the trace field of $\Gamma_\calO$ is $k$; recall that $\Gamma_\calO$ is defined to be $P\rho(\calO^1)$. It is similarly clear that the quaternion algebra of $\Gamma_\calO$ is $B$. Namely, since this algebra is a quaternion algebra over $k$ that is visibly contained in $B$, the asserted isomorphism follows from comparing dimensions. By definition, a Fuchsian group is arithmetic if it is commensurable with a group of the form $\Gamma_\calO$, hence by Lemma \ref{lem:Bdiscbound} and the preceding discussion, to prove Corollary \ref{cor:area1cor2} it suffices to bound the number of quaternion division algebras $B$ over $k$ which are ramified at all real places of $k$ and satisfy $\abs{\disc(B)}\leq \left[10^{93}V^{13}\right]^{10}$. The corollary now follows from Theorem \ref{thm:area1}.
\end{proof}

The proof of Corollary \ref{cor:area1cor1} is similar and is left to the reader.

\subsection{Lengths of geodesics arising from quadratic extensions}

We begin this subsection with a result that will permit us to work with Kleinian groups derived from a quaternion algebra.

\begin{prop}\label{prop:derivedindex}
Let $\Gamma$ be a Kleinian group with covolume $V$ and let $\Gamma^{(2)}$ be the subgroup of $\Gamma$ generated by squares. Then there exists an absolute, effectively computable constant $C$ such that the covolume of $\Gamma^{(2)}$ is at most $e^{CV}$.
\end{prop}

\begin{proof}
It is well-known that as $\Gamma$ has finite covolume it is finitely generated. Let $d(\Gamma)$ denote the minimal number of generators of $\Gamma$. By Theorem \ref{RankVolumeInequality}, there exists an absolute constant $C$ such that $d(\Gamma)<C_0V$. Now $\Gamma/\Gamma^{(2)}$ is a finite elementary abelian $2$--group of order at most $2^{d(\Gamma)}$. As the covolume of $\Gamma^{(2)}$ is $[\Gamma:\Gamma^{(2)}]\cdot V$, the result follows.
\end{proof}

\begin{lem}\label{lem:gammaindexbound}
Let $\Gamma^{\prime}$ be an arithmetic Fuchsian or Kleinian group derived from a quaternion algebra $B/k$ which has covolume $V^\prime$ and is contained in $\Gamma_\calO$, where $\calO$ is a maximal order of $B$. Then $[\Gamma_\calO:\Gamma]\leq V^\prime$.
\end{lem}

\begin{proof}
We prove the lemma in the case that $\Gamma^\prime$ is an arithmetic Kleinian group. The proof in the Fuchsian case is virtually identical. By Borel \cite{borel-commensurability} (see also \cite[Ch 11]{MR}), the covolume $V_{\calO}$ of $\Gamma_\calO$ is equal to
\begin{equation}
\frac{d_k^{3/2}\zeta_k(2)\prod_{\pp\mid \disc(B)} (\abs{\pp}-1)}{(4\pi^2)^{n_k-1}}.
\end{equation}
As $V_{\calO}\cdot [\Gamma_\calO:\Gamma^\prime] = V^\prime $ we see that
\begin{equation}\label{eq:1}
[\Gamma_\calO:\Gamma] = \frac{V^\prime(4\pi^2)^{n_k-1}}{d_k^{3/2}\zeta_k(2)\prod_{\pp\mid \disc(B)} (\abs{\pp}-1)}  \leq \frac{V^\prime(4\pi^2)^{n_k-1}}{d_k^{3/2}}.
\end{equation}
The discriminant bounds of Odlyzko \cite{Odlyzko-bounds} and Poitou \cite{Poitou} (see also \cite[Thm 2.4]{doud}) show that $\log(d_k)\geq 4r_1+6r_2$ where $r_1$ is the number of real places of $k$ and $r_2$ is the number of complex places of $k$. It is well-known that in our situation it must be the case that $k$ has a unique complex place, hence $r_1=n_k-2$ and $r_2=1$. We conclude that $d_k^{3/2}\geq e^{6(n_k-2)}e^9$. Applying this bound to equation (\ref{eq:1}) and simplifying finishes the proof.
\end{proof}

\begin{rmk}
It is implicit in the statement of Lemma \ref{lem:gammaindexbound} and in any event follows from the ideas of the lemma's proof that if $\calO$ is a maximal order of $B$ then $\Gamma_\calO$ has covolume at least $1$.
\end{rmk}	

We next need a simple lemma that provides a bound for the regulator of a maximal subfield.

\begin{lem}\label{lem:regbound}
Let $k$ be a number field with a unique complex place, $B/k$ a quaternion algebra which is ramified at all real places of $k$ and $L$ a maximal subfield of $B$. Then $\Reg_L\leq d_L^{n_k}$.
\end{lem}

\begin{proof}
Let $r_1(L)$ (resp., $r_2(L)$) be the number of real (resp., complex) places of $L$. As $L$ embeds into $B$, the Albert--Brauer--Hasse--Noether theorem implies that $r_1(L)=0$, $r_2(L)=n_k$. The class number formula \cite[p.~300]{Lang-ANT} yields $\Reg_L=\frac{\omega_Ld_L^{\frac{1}{2}}\kappa_L}{(2\pi)^{n_k}h_L}$, where $\omega_L$ is the number of roots of unity lying in $L$, $\kappa_L$ is the residue at $s=1$ of the Dedekind zeta function $\zeta_L(s)$ and $h_L$ is the class number of $L$. As $h_L\geq 1$, $\omega_L\leq 2n_L^2=8n_k^2$ and $8n^2\leq (2\pi)^n$ for all $n\geq 2$, we see that
\begin{align*}
\Reg_L &\leq d_L^{\frac{1}{2}}\kappa_L \leq d_L^{\frac{1}{2}}\log(d_L^{\frac{1}{2}})^{n_L-1} \leq d_L^{\frac{1}{2}}d_L^{n_k-\frac{1}{2}} = d_L^{n_k},
\end{align*}
where the second inequality follows from \cite[Prop 2]{Louboutin}.
\end{proof}

We will also need the following analog of Lemma \ref{lem:regbound}, whose proof is virtually identical to that of Lemma \ref{lem:regbound}.

\begin{lem}\label{lem:regbound2}
Let $k$ be a totally real number field, $B/k$ a quaternion algebra which is ramified at all but one real places of $k$ and $L$ a maximal subfield of $B$ which is not totally complex. Then $\Reg_L\leq d_L^{n_k}$.
\end{lem}

We briefly survey some basic results about logarithmic heights of algebraic numbers. For a number field $k$ and $\pp\in \Pp_k$, we normalize the associated valuation $\abs{\cdot}_\pp$ in the usual way so that for each $\alpha$ in $k$, we have
\[ \prod_{\pp\mid\infty} \abs{\alpha}_\pp=\abs{\Norm_{k/\Q}(\alpha)} \]
and $\prod_{\pp\in \Pp_k}\abs{\alpha}_\pp=1$. We define the \textbf{logarithmic height of $\alpha$ relative to $k$} to be
\[ h_k(\alpha)=\sum_{\pp\in \Pp_k} \log\pr{\max\set{1,\abs{\alpha}_\pp}}. \]
The \textbf{absolute height of $\alpha$} is $H(\alpha)=[k:\Q]^{-1}h_k(\alpha)$ and is independent of the field $k$. We remark that the \textbf{height of $\alpha$ relative to $\Q(\alpha)$} is the logarithm of the Mahler measure of the minimal polynomial of $\alpha$. We also remark the the height of $\alpha$ can be computed using only infinite places. The proof of the following lemma is straightforward.

\begin{lem}\label{lem:heights}
Let the notation be as above.
\begin{enumerate}
\item For all nonzero $n\in\Z$, $H(\alpha^n)=\abs{n}\cdot H(\alpha)$.
\item For all algebraic numbers $\beta$, $H(\alpha\beta)\leq H(\alpha)+H(\beta)$.
\item If $\alpha$ and $\beta$ are Galois conjugates then $H(\alpha)=H(\beta)$.
\end{enumerate}
\end{lem}

\begin{prop}\label{proposition:derivedcovolume}
Let $\Gamma$ be an arithmetic Fuchsian or Kleinian group which has covolume $V$ and with trace field $k$ and quaternion algebra $B$. Let $L/k$ be a quadratic extension which embeds into $B$ and suppose further that $L$ is not totally complex if $k$ is totally real. Then there exist absolute, effectively computable constants $C_1,C_2$ such that $L=k_\gamma$ for some hyperbolic $\gamma\in\Gamma$ with length at most $e^{C_1V}d_L^{C_2+\log(V)}$.
\end{prop}

\begin{proof}
We prove the proposition in the case that $k$ has a unique complex place. The case in which $k$ is totally real has an identical proof. As every real place of $k$ is ramified in $B$, we see that $L$ embeds into $B$ implies, by the Albert--Brauer--Hasse--Noether theorem, that $L$ is totally complex. By Dirichlet's unit theorem, the $\Z$--rank of $\calO_L^*$ is strictly greater than that of $\calO_k^*$. From this we conclude that every system of fundamental units of $\calO_L^*$ contains a fundamental unit $u_0\in\calO_L^*$ such that $u_0^n\notin k$ for any $n\geq 1$. Hence, we have $L=k(u^n)$ for all $n\neq 0$. Let $\sigma$ denote the non-trivial automorphism of $\Gal(L/k)$ and define $u=u_0/\sigma(u_0)$. It is clear that $\Norm_{L/k}(u)=1$ and that $u^n\notin\calO_k^*$ for any $n\geq 1$. By \cite{hajdu} (see also \cite{brindza}) we may take $u_0$ to have logarithmic height (relative to $L$) $h_L(u_0)\leq n_k^{11n_k}\Reg_L$. It follows from Lemma \ref{lem:heights} that $h_L(u)\leq 2n_k^{11n_k}\Reg_L$, and since $[L:\Q]=2n_k$ we see as well that $H(u)\leq n_k^{11n_k-1}\Reg_L$.

As $\Gamma^{(2)}$ is derived from the quaternion algebra $B$ \cite[Ch 3]{MR}, there exists a maximal order $\calO$ of $B$ such that $\Gamma^{(2)}\subset \Gamma_{\calO}^1$. Recall that $k_B$ is the maximal abelian extension of $k$ of exponent $2$ in which all prime divisors of $\disc(B)$ split completely. We have two cases. Suppose first that $L\not\subset k_B$. Then every maximal order of $B$ admits an embedding of $\calO_k[u]$ (\cite[Thm 3.3]{Chinburg-Friedman-selectivity}; see also \cite[Prop 5.4]{linowitz-selectivity}), hence we may assume that $u\in\calO$. Let $\gamma^\prime$ be the image in $\Gamma_{\calO}^1$ of $u$. Proposition \ref{prop:derivedindex} and Lemma \ref{lem:gammaindexbound} show that $\gamma={\gamma^\prime}^n\in\Gamma^{(2)}\subset \Gamma$ for some $n\leq e^{C_0V}$ and constant $C_0$. By Lemma \ref{lem:heights} we have $H(\gamma)\leq e^{C_0V}n_k^{11n_k-1}\Reg_L$. As the logarithm of the Mahler measure of the minimal polynomial of $\gamma$ is less than $2n_kH(\gamma)$, by \cite[Lemma 12.3.3]{MR} we have $\ell_0(\gamma)\leq 4e^{C_0V}n_k^{11n_k}\Reg_L$. By Lemma \ref{lem:regbound}, we have $\ell_0(\gamma)\leq 4e^{C_0V}n_k^{11n_k}d_L^{n_k}$. By \cite[Lemma 4.3]{chinburg-smallestorbifold}, we have $n_k\leq 23+\log(V)$. Hence there exists a constant $C_1$ such that $\ell_0(\gamma)\leq e^{C_1V}d_L^{23+\log(V)}$.

Suppose now that $L\subset k_B$ and $\calE$ is a maximal order of $B$ containing $u$. By Proposition \ref{proposition:selectivecase} there exists an absolute constant $C_2>0$ and an integer $n\leq d_L^{C_2}$ such that $u^n\in\calO^1$. The arguments of the previous paragraph show that there exists $\gamma\in\Gamma^{(2)}\subset \Gamma$ with length at most $e^{C_1V}d_L^{C_2+\log(V)}$, finishing our proof.
\end{proof}

\begin{rmk}
Proposition \ref{proposition:derivedcovolume} is an effective version of Theorem 12.2.6 of \cite{MR}.	
\end{rmk}

\subsection{Proof of Corollary \ref{cor:GeoCount}}


\begin{thm}\label{thm:area2app1}
Let $k$ be a number field of degree $n_k$, discriminant $d_k$ that is totally real (resp., has a unique complex place). Let $B/k$ be a quaternion algebra which is ramified at all but one real places of $k$ (resp., at all real places of $k$) and let $\calO$ be a maximal order of $B$. Then for all sufficiently large $x$, the orbifold $\textbf{H}^2/\Gamma_\calO$ (resp., $\textbf{H}^3/\Gamma_\calO$) contains at least $\left[\frac{\kappa_k}{2}\left(\frac{3}{\pi^2}\right)^{n_k}\right]x$ rationally inequivalent geodesics of length at most $\left[2n_k^{11n_k-1}d_k^{2n_k}\right]x^{n_k}$.
\end{thm}

\begin{proof}
By Theorem \ref{thm:countingquads} and the well-known fact that $\zeta_k(s)\leq \zeta(s)^{n_k}$, for all sufficiently large real $x>0$ there are at least $\left[2\kappa_k\left(\frac{3}{\pi^2}\right)^{n_k}\right]x$ quadratic extensions $L/k$ which embed into $B$ and satisfy $\abs{\Delta_{L/k}}<x$. When $k$ is totally real, if an extension $L/k$ embeds into $B$ then $L$ is either totally complex or else has $2$ real places and $n_k-2$ complex places. Combining Theorem \ref{thm:countingquads} with Proposition \ref{prop:wood}(iii) now shows that when $k$ is totally real there are at least $\left[\kappa_k\left(\frac{3}{\pi^2}\right)^{n_k}\right]x$ quadratic extensions $L/k$ which embed into $B$, satisfy $\abs{\Delta_{L/k}}<x$, and are not totally complex. The proof of Proposition \ref{proposition:derivedcovolume} shows that with at most finitely many exceptions, the extensions described above are all of the form $L=k_\gamma$, where $\gamma \in \Gamma_\calO$ is hyperbolic with length at most $2n_k^{11n_k-1}\Reg_L.$ By \cite[Lemma 6.3]{chinburg-geodesics}, if $\lambda_1,\lambda_2$ are eigenvalues of hyperbolics $\gamma_1,\gamma_2\in\Gamma_\calO$ whose complex lengths are rationally equivalent then either $k(\lambda_1)=k(\lambda_2)$ or $k(\lambda_1)=k(\overline{\lambda_2})$. It follows that for $x$ sufficiently large, at least $\left[\frac{\kappa_k}{2}\left(\frac{3}{\pi^2}\right)^{n_k}\right]x$ of the geodesics associated to the hyperbolic $\gamma$ are rationally inequivalent. The theorem now follows from Lemma \ref{lem:regbound} and $d_L=\abs{\Delta_{L/k}}d_k^2$.
\end{proof}


\begin{proof}[Proof of Corollary \ref{cor:GeoCount}]
By Proposition \ref{prop:derivedindex} (which follows from Gauss--Bonnet when $\Gamma$ is a Fuchsian group) the covolume of $\Gamma^{(2)}$ is at most $e^{CV}$ for some absolute, effectively computable constant $C$. It is well-known \cite[Cor 8.3.5]{MR} that there is a maximal order $\calO$ of $B$ such that $\Gamma^{(2)}\leq \Gamma_\calO$, and it was shown in Lemma \ref{lem:gammaindexbound} that the index of $\Gamma^{(2)}$ in $\Gamma_\calO$ is at most $e^{CV}$. It follows immediately from Theorem \ref{thm:area2app1} that for $x$ sufficiently large, the orbifold $\textbf{H}^3/\Gamma$ contains at least $\left[\frac{\kappa_k}{2}\left(\frac{3}{\pi^2}\right)^{n_k}\right]x$ rationally inequivalent geodesics of length at most $e^{CV}(2n_k^{11n_k-1}d_k^{2n_k})x^{n_k}.$
The corollary now follows from $n_k\leq 23+\log(V)$ \cite[Lemma 4.3]{chinburg-smallestorbifold}, $d_k\leq V^{22}$ \cite[proof of Thm 4.1]{linowitz-isospectralsize}, and logarithm inequalities.
\end{proof}

\subsection{Counting manifolds with prescribed geodesic lengths}

Let $M$ be an arithmetic hyperbolic $2$--orbifold (resp., $3$--orbifold) with $\pi_1(M) = \Gamma$ that has geodesics with lengths (resp., complex length) $\ell_1,\dots,\ell_N$. For $V>0$, $N^2_{\ell_1,\dots,\ell_N}(V)$ (resp., $N^3_{\ell_1,\dots,\ell_N}(V)$) is the maximum cardinality of a family of arithmetic, pairwise non-commensurable, hyperbolic $2$--orbifolds (resp., $3$--orbifolds) all of which have geodesics with lengths (resp., complex lengths) $\ell_1,\dots,\ell_N$ and volume at most $V$. By Borel \cite[Thm 8.2]{borel-commensurability}, both $N^2_{\ell_1,\dots,\ell_N}(V),N^3_{\ell_1,\dots,\ell_N}(V)$ finite. As an application of Theorem \ref{thm:quatembedding} we provide lower and upper bounds for these functions. \cite{LMP} establish estimates for similarly defined counting functions.

Before proceeding, we fix some notation which we will use for the remainder of this section. With $\ell_1,\dots,\ell_N$ as above, let $\gamma_i \in \Gamma$ be hyperbolic with associated geodesic of length $\ell_i$ and let $\lambda_i$ denote the eigenvalue of a pre-image of $\gamma_i$ in $\SL(2,\R),\SL(2,\C)$ satisfying $\abs{\lambda_i}>1$.

We now state our bounds for $N^2_{\ell_1,\dots,\ell_N}(V)$ and $N^3_{\ell_1,\dots,\ell_N}(V)$. Theorem \ref{thm:area3app1} deals with the case in which $\{\lambda_1,\dots,\lambda_N\}\not\subset \R$ and Theorem \ref{thm:area3app2} deals with the case in which $\{\lambda_1,\dots,\lambda_N\}\subset \R$.

\begin{thm}\label{thm:area3app1}
Let $M$ be an arithmetic hyperbolic $3$--manifold which is derived from a quaternion algebra and has geodesics with complex lengths $\ell_1,\dots,\ell_N$. If $\lambda_i$ is not real for some $i$, then exactly one of the following is true:
\begin{enumerate}
\item There are only finitely many quaternion algebras defined over the trace field $k$ of $\Gamma$ which are ramified at all real places of $k$ and admit embeddings of $k(\lambda_i)$ for all $i$. In this case there are positive real numbers $c$ and $V_0$ such that $N^3_{\ell_1,\dots,\ell_N}(V)=c$ for all $V>V_0$.
\item There are infinitely many commensurability classes of hyperbolic $3$--orbifolds that contain an orbifold that has geodesics with complex lengths $\ell_1,\dots,\ell_N$. In this case there exist integers $1\leq r,s\leq N$ such that
\[ V/\log(V)^{1-\frac{1}{2^s}} \ll N^3_{\ell_1,\dots,\ell_N}(V)\ll V/\log(V)^{1-\frac{1}{2^r}}, \] where the implicit constants depend only on $k$ and ${\ell_1,\dots,\ell_N}$.
\end{enumerate}
\end{thm}

\begin{rmk}
We remark that \cite[Lemma 2.3]{chinburg-geodesics} shows that if $k$ is not a quadratic extension of $k^+$ then $\Gamma$ will have no hyperbolics with real eigenvalue. In this situation the hypotheses of Theorem \ref{thm:area3app1} will always be satisfied.
\end{rmk}

The techniques used to prove Theorem \ref{thm:area3app1} can be applied, in much the same manner, to prove the following result.

\begin{thm}\label{thm:area3app2}
Let $M$ be an arithmetic hyperbolic $2$--manifold (resp., $3$--manifold) which is derived from a quaternion algebra and contains geodesics with lengths (resp., complex lengths) $\ell_1,\dots,\ell_N$. If $\lambda_i$ is real for all $i$, then exactly one of the following is true:
\begin{enumerate}
\item There are only finitely many quaternion algebras defined over the trace field $k$ of $\Gamma$ which are ramified at all but one real places of $k$ (resp., at all real places of $k$) and admit embeddings of $k(\lambda_i)$ for all $i$.
\item There are infinitely many commensurability classes of hyperbolic $2$--orbifolds (resp., $3$--orbifolds) that contain an orbifold that has geodesics with lengths (resp., complex lengths) $\ell_1,\dots,\ell_N$ and trace field $k$. In this case there exist integers $1\leq r,s,t\leq N$ such that
\begin{align*}
V/\log(V)^{1-\frac{1}{2^r}} &\gg N^2_{\ell_1,\dots,\ell_N}(V)\gg V/\log(V)^{1-\frac{1}{2^s}}\qquad (\mbox{resp., } N^3_{\ell_1,\dots,\ell_N}(V) \gg V/\log(V)^{1-\frac{1}{2^t}}).
\end{align*}
\end{enumerate}
\end{thm}

Note that Theorems \ref{thm:area3app1} and \ref{thm:area3app2} both deal with geodesics on arithmetic hyperbolic $2$-- and $3$--manifolds which are derived from quaternion algebras, as opposed to arbitrary arithmetic hyperbolic $2$-- and $3$--manifolds. The reason for this restriction amounts to the following observation (which will be made more precise and proven as part of the proof of Theorem \ref{thm:area3app1}). Let $M$ be as in Theorem \ref{thm:area3app1}. The commensurability classes of arithmetic hyperbolic $3$--orbifolds that contain an orbifold that has geodesics with complex lengths $\ell_1,\dots,\ell_N$ are in one-to-one correspondence with quaternion algebras over $k$ which ramify at all real places of $k$ and admit embeddings of $k(\lambda_1),\dots,k(\lambda_N)$. This correspondence breaks down however, when the manifold $M$ is arithmetic but not necessarily derived from a quaternion algebra. In this more general setting however, we are able to prove the following.

\begin{thm}\label{thm:area3app3}
Let $M$ be an arithmetic hyperbolic $2$--manifold (resp., $3$--manifold) that has geodesics with lengths (resp., complex lengths) $\ell_1,\dots,\ell_N$. If there are infinitely many primes of $k$ which do not split in any of the extensions $k(\lambda_i)/k$ then there are infinitely many commensurability classes of hyperbolic 2--orbifolds (resp., $3$--orbifolds) that contain an orbifold that has geodesics with lengths (resp., complex lengths) $\ell_1,\dots,\ell_N$.
\end{thm}

\begin{rmk}
If there are only finitely many primes of $k$ which do not split in any of the extensions $k(\lambda_i)/k$, then there are at most finitely many commensurability classes of hyperbolic $2$-- or $3$--orbifolds that contain an orbifold that has geodesics with (complex) lengths $\ell_1,\dots,\ell_N$ \textit{and} trace field $k$. In many situations however (for instance if $M$ is a $2$-- or $3$--manifold such that $\{ \lambda_1,\dots,\lambda_N\}\not\subset \R$), any arithmetic hyperbolic $2$--orbifold (resp., $3$--orbifold) that has geodesics with lengths (resp., complex lengths) $\ell_1,\dots,\ell_N$ must have trace field $k$. See for instance Proposition \ref{prop:fuchsiantracefield} and \cite[Lemma 2.3]{chinburg-geodesics}.  In these situations the hypothesis in Theorem \ref{thm:area3app3} is necessary and sufficient.
\end{rmk}

\subsubsection{Proof of Theorem \ref{thm:area3app1}}

We begin with a proposition that will be needed in the proof of Theorem \ref{thm:area3app1}.

\begin{prop}\label{proposition:alwaysexistsderived}
Let $B/k$ be a quaternion algebra which admits embeddings of $k(\lambda_1),\dots,k(\lambda_N)$ and $\calO \su B$ be a maximal order. If $\Ram_f(B) \ne \emptyset$, then the orbifold associated to $\Gamma_\calO$ has geodesics with (complex) lengths $\ell_1,\dots,\ell_N$.
\end{prop}

\begin{proof}
For each $i=1,\dots,N$, fix a quadratic $\calO_k$--order $\Omega_i\subset k(\lambda_i)$ which contains a pre-image in $k(\lambda_i)$ of $\gamma_i$. As $B$ ramifies at a finite prime of $k$, by \cite[Thm 3.3]{Chinburg-Friedman-selectivity}, every maximal order of $B$, in particular $\calO$, contains a conjugate of all of the quadratic orders $\Omega_i$. The proposition now follows from the fact that the (complex) length of the geodesic associated to $\gamma_i$ coincides with the (complex) length of the geodesic associated to any conjugate of $\gamma_i$.
\end{proof}

\begin{proof}[Proof of Theorem \ref{thm:area3app1}]
Let $k,B$ denote the trace field and quaternion algebra of $\Gamma$ and for $i=1,\dots,N$, let $L_i$ denote the quadratic extension $k(\lambda_i)$ of $k$. By hypothesis there exists an $i$ such that $\lambda_i\notin \R$. By \cite[Lemma 2.3]{chinburg-geodesics}, the image in $\C$ of $k$ is $\Q(\tr(\gamma_i))=\Q(\lambda_i+\lambda_i^{-1})$. Throughout the remainder of this proof we will identify $k$ with its image in $\C$. Suppose that $\Gamma^\prime$ is an arithmetic Kleinian group such that the quotient orbifold has geodesics with complex lengths $\ell_1,\dots,\ell_N$. Taking powers of the elements $\gamma_i$ as needed, we may assume that $\Gamma^\prime$ is derived from a quaternion algebra. Let $k^\prime,B'$ denote the trace field and quaternion algebra of $\Gamma^\prime$. By (\ref{TraceLengthFormula}), if $\gamma_i^\prime \in \Gamma^\prime$is hyperbolic with associated geodesic of complex length $\ell_i$, then $\tr(\gamma_i)=\pm \tr(\gamma_i^\prime)$. In particular, up to complex conjugation, we have $k=\Q(\tr(\gamma_i))=\Q(\tr(\gamma_i^\prime))=k^\prime$. We may therefore suppose that $B^\prime$ is defined over $k$. The results of \cite[Ch 12]{MR} now imply that $B^\prime$ admits embeddings of $L_1,\dots,L_N$. Conversely, suppose that $B^\prime/k$ is a quaternion algebra which satisfies the following two conditions:

\begin{enumerate}
\item $B^\prime$ is ramified at all real places and at at least one finite prime of $k$,
\item $B^\prime$ admits embeddings of $L_1,\dots,L_N$.
\end{enumerate}

Proposition \ref{proposition:alwaysexistsderived} then shows that if $\calO^\prime$ is a maximal order of $B^\prime$ then $\Gamma_{\calO^\prime}$ is an arithmetic Kleinian group whose quotient orbifold has geodesics with complex lengths $\ell_1,\dots,\ell_N$. Putting these together, we see that $N_{\ell_1,\dots,\ell_N}(V)$ is asymptotic to the number of isomorphism classes of quaternion algebras over $k$ which are ramified at all real places of $k$ and which admit embeddings of $L_1,\dots,L_N$. (Note that all but finitely many quaternion algebras over $k$ ramify at a finite prime of $k$.) The first assertion is an immediate consequence of this.

In order to prove the second assertion we first show that $N^3_{\ell_1,\dots,\ell_N}(V)\ll V/\log(V)^{1-\frac{1}{2^r}}$ for some $1\le r \le N$. Suppose that $\Gamma^\prime$ is an arithmetic Kleinian group whose quotient orbifold has geodesics with complex lengths $\ell_1,\dots,\ell_N$ and let $V^\prime$ denote the covolume of $\Gamma^\prime$. If we let $V_{\mathscr C}$ denote the volume of a minimal volume orbifold in the commensurability class $\mathscr C$ of $\Gamma$, then $V^\prime\geq V_{\mathscr C}$. Borel's formula \cite{borel-commensurability} for $V_{\mathscr C}$ makes it clear that there exists a constant $c$, which depends on $k$, such that $V_{\mathscr C}\geq c \abs{\disc(B^\prime)}$ where $B^\prime$ is the quaternion algebra of $\Gamma^\prime$ and $\abs{\disc(B^\prime)}$ the norm of its discriminant. It follows from the discussion above that $B^\prime$ is defined over $k$ and it is clear that $B^\prime$ admits embeddings of $L_1,\dots,L_N$. Hence, the number of commensurability classes of arithmetic hyperbolic $3$--orbifolds that contain an orbifold that has geodesics with complex lengths $\ell_1,\dots,\ell_N$ is at most a constant multiple of the number of quaternion algebras over $k$ which admit embeddings of $L_1,\dots,L_N$. As $\abs{\disc(B^\prime)}\leq cV^\prime$, that $N^3_{\ell_1,\dots,\ell_N}(V)\ll V^{1/2}/\log(V)^{1-\frac{1}{2^r}}$ for some $1\le r \le N$ now follows from Theorem \ref{thm:quatembedding}. The proof that $N^3_{\ell_1,\dots,\ell_N}(V)\gg V/\log(V)^{1-\frac{1}{2^s}}$ for some $1\leq s\le N$ follows from the same ideas though applied to orbifolds of the form considered in Proposition \ref{proposition:alwaysexistsderived}.
\end{proof}

\subsubsection{Remarks about the proof of Theorem \ref{thm:area3app2}}\label{subsection:area3app2proof}

The proof of Theorem \ref{thm:area3app2} follows from the same arguments that were used to prove the analogous statements in Theorem \ref{thm:area3app1}, hence we omit it. We do, however, record the following proposition which serves as a substitute for Lemma 2.3 of \cite{chinburg-geodesics} in the Fuchsian case.

\begin{prop}\label{prop:fuchsiantracefield}
Let $\Gamma$ be an arithmetic Fuchsian group derived from a quaternion algebra $B/k$. If $\gamma\in\Gamma$ is a hyperbolic with eigenvalue $\lambda_\gamma$ then $k=\Q(\tr(\gamma))$.
\end{prop}


\begin{proof}
Let $\Gamma_0$ be an arithmetic Kleinian group derived from a quaternion algebra which contains $\Gamma$ and whose trace field $k_0$ is a quadratic extension of $k$. Note, the existence of such a group $\Gamma_0$ follows from the results in \cite[Ch 9]{MR}. Set $F=\Q(\tr(\gamma))$. Since $\gamma\in\Gamma_0$ and $\lambda_\gamma\in\R$, \cite[Lemma 2.3]{chinburg-geodesics} shows that $[k_0:F]=2$ and that $F$ is the maximal totally real subfield of $k_0$. It is now clear that $F=k$, completing the proof.
\end{proof}

We now make a few comments about why the techniques used to prove Theorem \ref{thm:area3app1} do not suffice to prove an upper bound for $N^3_{\ell_1,\dots,\ell_N}(V)$. The upper bound for $N^3_{\ell_1,\dots,\ell_N}(V)$ in Theorem \ref{thm:area3app1} relied upon the fact that any arithmetic hyperbolic $3$--orbifold that has geodesics with complex lengths $\ell_1,\dots,\ell_N$ necessarily has $k$ as its trace field. Hence, it is obtained from counting quaternion algebras over $k$ admitting embeddings of $k(\lambda_1),\dots,k(\lambda_N)$. Whereas this is the case for arithmetic Fuchsian groups by Proposition \ref{prop:fuchsiantracefield}, it is not necessarily the case for $3$--orbifolds in the context of Theorem \ref{thm:area3app2}. Indeed, let $k^+$ denote the maximal totally real subfield of $k$ and assume that $k$ is a quadratic extension of $k^+$. Lemma 2.3 of \cite{chinburg-geodesics} shows that $k=\Q(\tr(\gamma_i))$ if $\lambda_i$ is not real and $k^+=\Q(\tr(\gamma_i))$ if $\lambda_i$ is real. As a consequence the trace field of an arithmetic hyperbolic $3$--orbifold that has geodesics with complex lengths $\ell_1,\dots,\ell_N$ is a quadratic extension of $k^+$. This does not imply that this trace field is equal to $k$. In theory one could obtain an upper bound for $N^3_{\ell_1,\dots,\ell_N}(V)$ by counting the number of quadratic extensions of $k^+$ with norm of relative discriminant less than some bound and having a unique complex place and then multiplying this count by the number of quaternion algebras defined over each field. The former count has been computed by Cohen--Diaz y Diaz--Olivier \cite[Corollary 3.14]{quadraticextensions}. The latter count is given by Theorem \ref{thm:quatembedding} and contains a constant which depends on the invariants of the particular quadratic extension of $k^+$ chosen. It is not clear how one could bound this constant with invariants of only $k^+$ due to the complexity of this constant. Those invariants also need to be directly related to the volume of the $3$--orbifold. Another difficulty is the complexity of the error terms implicit in Theorem \ref{thm:quatembedding}. Specifically, Theorem \ref{thm:quatembedding} requires that $x\to\infty$, and the rate at which $x\to\infty$ that is needed could vary along with the quadratic extension of $k^{+}$.

\subsubsection{Proof of Theorem \ref{thm:area3app3}}

We prove the theorem in the case in which $M$ is a $3$--manifold. The surface case has a proof which is virtually identical and thus left to the reader. Let $\Gamma$ be the fundamental group of $M$ and $B$ the associated quaternion algebra. We may assume without loss of generality that $\Gamma$ is a maximal arithmetic subgroup of $B^\times/k^\times$ and hence is of the form $\Gamma=\Gamma_{S,\calD}$ (in the notation of \cite[\S 4]{Chinburg-Friedman-selectivity}), where $S$ is a finite set of primes of $k$ and $\calD$ is a maximal order of $B$. For $i=1,\dots,N$, let $\overline{y_i}\in \Gamma_{S,\calD}$ be a pre-image of $\gamma_i$ in $B^\times/k^\times$ and $y_i\in k(\lambda_i)$ be a pre-image of $\gamma_i$ in $B^\times$. In order to prove the existence of infinitely many pairwise non-commensurable arithmetic hyperbolic $3$--orbifolds that contain geodesics with complex lengths $\ell_1,\dots,\ell_N$, we first construct an infinite number of quaternion algebras $B_j/k$ (each of which ramifies at all real places of $k$) with the property that for every $j$, the finite part of $\disc(B)$ is a proper divisor of the finite part of $\disc(B_j)$. We will then construct, for every $j$, a maximal arithmetic subgroup $\Gamma_j$ of $B_j^\times/k^\times$ such that $\Gamma_j$ contains, for $i=1,\dots,N,$ an element with the same trace and norm as $\overline{y_i}$. It will follow that the associated orbifold will have geodesics with complex lengths $\ell_1,\dots,\ell_N$.

Our construction of the algebras $B_j$ is straightforward. Let $\pp_1,\pp_2,\dots$ be an infinite sequence of primes of $k$ which do not split in any of the extensions $k(\lambda_i)/k$. Pruning this sequence as needed, we may assume that none of the primes $\pp_i$ lie in the finite set $S$ of primes mentioned in the previous paragraph nor do they divide $\disc(B)$. Consider the sequence of moduli $\set{\disc(B)\pp_1\pp_j}_{j>1}$. As $\disc(B)$ must have an even number of divisors, as do the discriminants of all quaternion algebras over number fields, each of these moduli has an even number of divisors. Hence, there exist quaternion algebras $B_1,B_2,B_3\dots$ having these as their discriminants and these algebras are pairwise non-isomorphic. Also note that by the Albert--Brauer--Hasse--Noether theorem, the quadratic extension $k(\lambda_i)/k$ will embed into $B_j$ if and only if no prime which ramifies in $B_j$ splits in $k(\lambda_i)/k$. As the extension $k(\lambda_i)/k$ embeds into $B$, none of the divisors of $\disc(B)$ split in any of the extensions $k(\lambda_i)/k$. Further, by hypothesis no prime in the sequence $\pp_1,\pp_2,\dots$ splits in any of the extensions $k(\lambda_i)/k$. We conclude that all of the algebras $B_j$ admit embeddings of all of the extensions $k(\lambda_i)/k$.

To construct the maximal arithmetic subgroups $\Gamma_j$ of $B_j^\times/k^\times$, we need the following result from \cite[Thm 4.4]{Chinburg-Friedman-selectivity}:

\begin{thm}[Chinburg--Friedman]\label{thm:maximalselectivity}
Let $k$ be a number field and $B/k$ be a quaternion algebra in which at least one archimedean place of $k$ is unramified. Suppose that $y\in B^\times$ and consider the maximal arithmetic subgroup $\Gamma_{S,\calD}$ of $B^\times/k^\times$. If a conjugate of the image $\overline{y}\in B^\times/k^\times$ of $y$ is contained in $\Gamma_{S,\calD}$ then the following three conditions hold:

\begin{enumerate}
\item $\disc(y)/\Norm(y)\in\calO_k$,
\item If an odd power of $\pp$ appears in the prime factorization of $n(y)$ ($y$ is odd at $\pp$), then $\pp\in S\cup \Ram_f(B)$,
\item For each $\pp\in S$ at least one of the following four conditions hold:
\begin{enumerate}
\item $y\in k$;
\item $y$ is odd at $\pp$;
\item $k(y)\otimes_k k_\pp$ is not a field;
\item $\pp$ divides $\disc(y)/\Norm(y)$.
\end{enumerate}
\end{enumerate}

Conversely, if conditions (1), (2) and (3) hold, then a conjugate of $\overline{y}$ is contained in $\Gamma_{S,\calD}$ except possibly when the following three conditions hold:

\begin{enumerate}
\setcounter{enumi}{3}
\item $k(y)\subset B$ is a quadratic field extension of $k$.
\item The extension $k(y)/k$ and the algebra $B$ are both unramified at all finite primes of $k$ and ramify at precisely the same (possibly empty) set of real places of $k$. Further, all primes $\pp\in S$ split in $k(y)/k$.
\item All primes $\pp$ dividing $\disc(y)/\Norm(y)$ split in $k(y)/k$.
\end{enumerate}

Suppose now that (1)-(6) hold. In this case the number of $S$--types of maximal orders $\calD$ of $B$ is even and the $\calD$ for which a conjugate $\overline{y}$ belongs to $\Gamma_{S,\calD}$ comprise exactly half of the $S$--types.
\end{thm}

We now return to the proof of Theorem \ref{thm:area3app3}. Fix an integer $j\geq 1$ and consider the quaternion algebra $B_j$ defined above. Let $\calO \su B_j$ be a maximal order and consider the maximal arithmetic subgroup $\Gamma_j=\Gamma_{S,\calO}$ of $B_j^\times/k^\times$. As $B_j$ admits embeddings of $k(\lambda_i)/k$ for all $i$, we abuse notation and identify these extensions with their images in $B_j$. We view the $y_i$ above (in the context of the algebra $B$) as being contained in $B_j$. As the $y_i$ were all contained in $\Gamma_{S,\calD}\subset B^\times/k^\times$, we see by Theorem \ref{thm:maximalselectivity} (and because $B_j$ ramifies at a finite prime of $k$, hence condition (5) of Theorem \ref{thm:maximalselectivity} is not satisfied, and $\Ram_f(B)\subset \Ram_f(B_j)$) that conjugates of the $\overline{y_i}$ lie in $\Gamma_j\subset B_j^\times/k^\times$. Theorem \ref{thm:area3app3} follows.

\section{Proof of Theorem \ref{thm:effectiveCHLR}}\label{GeometricRigiditySection}

\subsection{A technical lemma}

Let $k$ be a number field of degree $n_k$ with integral basis $\Omega=\{\omega_1,\dots,\omega_{n_k}\}$. We endow $\calO_k$ with the $T_2$--norm by setting $T_2(x)=\sum_{\sigma: k\hookrightarrow \C} |\sigma(x)|^2$. An immediate consequence of the arithmetic-geometric mean inequality is that $T_2(x)\geq n_k$ for all $x\neq 0$. Define
\[ B(\Omega)=\prod_{\sigma: k\hookrightarrow \C} \sum_{i=1}^{n_k}|\sigma(\omega_i)|. \]

\begin{lem}\label{lem:integralbasisbound}
Let $k$ be a number field of degree $n_k\geq 2$. Then $B(\Omega)\leq 2^{n_k^3}d_k^{n_k}$.
\end{lem}
\begin{proof}
The proof follows from the following inequalities:
\begin{align*}
B(\Omega) & = \prod_{\sigma: k\hookrightarrow \C} \sum_{i=1}^{n_k}|\sigma(\omega_i)| \leq \prod_{\sigma: k\hookrightarrow \C} \sum_{i=1}^{n_k}T_2(\omega_i) \leq \prod_{\sigma: k\hookrightarrow \C} \prod_{i=1}^{n_k}T_2(\omega_i)  \leq \prod_{\sigma: k\hookrightarrow \C}\left(2^{n_k^2}d_k\right) \leq 2^{n_k^3}d_k^{n_k},
\end{align*}
where the second to last inequality follows from \cite[Thm 3]{t2normbound}.		
\end{proof}

\subsection{Geodesics of bounded length arising from maximal subfields}

\begin{prop}\label{proposition:derivednonrealeigenvalue}
Let $\Gamma$ be an arithmetic Kleinian group with trace field $k$, quaternion algebra $B$, and covolume $V$. Then there exists a hyperbolic $\gamma\in\Gamma$ with eigenvalue $\lambda=\lambda_\gamma$ such that $\lambda^n$ is not real for any $n\geq 1$ and $\gamma$ has length at most $Ke^{\left(\log(V)^{\log(V)}\right)}$ for some absolute constant $K$.
\end{prop}

\begin{proof}
To begin, we make effective an argument from \cite[p.~10]{chinburg-geodesics}. We start by proving the existence of a hyperbolic $\gamma\in\Gamma$ with eigenvalue $\lambda$ such that $\lambda^n$ is not real for any $n>1$ and has length at most $e^{C_1V}d_L^{C_2+\log(V)}$. We first consider the case in which $k/k^+$ is a quadratic extension. By the Chebotarev density theorem, there are infinitely many rational primes $p$ which split completely in $k/\Q$ and do not divide $\abs{\disc(B)}$. To obtain an upper bound we use a modification \cite[Thm 2-C]{effectiveGW} of the bound on the least prime ideal in the Chebotarev density theorem in \cite{effectiveCDT}:

\begin{thm}[Wang]\label{thm:effectiveCDT}
Let $L/K$ be a finite Galois extension of number fields of degree $n$, $S$ a finite set of primes of $K$ and $[\theta]$ a conjugacy class in $\Gal(L/K)$. Then there is a prime ideal $\pp$ of $K$ such that
\begin{enumerate}
	\item $\pp$ is unramified in L and is of degree $1$ over $\Q$;
	\item $\pp\notin S$;
	\item $\left(\frac{L/K}{\pp}\right)=[\theta]$, and
	\item $\abs{\pp}\leq d_L^C (n\log(N_S))^2$,
\end{enumerate}
where $C$ is an absolute, effectively computable constant and $N_S=\prod_{\qq\in S}\abs{\qq}$.
\end{thm}

We would like to apply Theorem \ref{thm:effectiveCDT} to the extension $k/\Q$ but cannot as $k/\Q$ need not be Galois. Let $\widehat{k}$ be the Galois closure of $k$. The extension $\widehat{k}/\Q$ is by definition Galois, has degree at most $n_k!$, and has the property that a prime $p$ of $\Q$ splits completely in (resp. ramifies in) $k$ if and only if $p$ splits completely in (resp. ramifies in) $\widehat{k}$. Moreover, Serre \cite[Prop 6]{Serre-CDT} shows that $d_{\widehat{k}}\leq d_k^{n_k!-1}{n_k!}^{n_k!}.$ Let $p$ be a fixed rational prime which splits completely in $k$ and does not lie below any prime ramifying in $B$. By Theorem \ref{thm:effectiveCDT} (applied with $K=\Q$ and $L=\widehat{k}$), we may assume that
\[ p\leq \left[d_k^{n_k!-1}{n_k!}^{n_k!}\right]^A\cdot \left[n_k!\log\abs{\disc(B)}\right]^2. \]
Let $\qq^+$ be a fixed prime of $k^+$ lying above $\pp$ and $\qq_1,\qq_2$ be distinct primes of $k$ lying above $\qq^+$. Let $L/k$ be a quadratic extension which is ramified at every prime divisor of $\disc(B)$ and furthermore satisfies that $\qq_1$ ramifies in $L/k$ and $\qq_2$ splits in $L/k$. Then there exist primes $\qq_1',\qq_2',\qq_3'$ of $L$ such that $\qq_1\calO_L=(\qq_1')^2$ and $\qq_2\calO_L=\qq_2'\qq_3'$.  By Proposition \ref{proposition:derivedcovolume} there exist absolute, effectively computable constants $C_1,C_2$ such that $L=k(\lambda(\gamma))$ for some hyperbolic $\gamma\in\Gamma$ with length at most $e^{C_1V}d_L^{C_2+\log(V)}$. As the primes $\qq_1'$ and $\qq_2'$ both lie above $\qq^+$ and have different ramification degrees, we can infer that the extension $L/k^+$ is not Galois, hence $\lambda(\gamma)$ is not real by \cite[Lemma 2.3]{chinburg-geodesics}. As $k(\lambda(\gamma))=k(\lambda(\gamma)^n)$ for all $n\geq 1$, our assertion that no power of $\lambda(\gamma)$ is real follows from an identical argument.

If $k/k^+$ is not quadratic, then \cite[Lemma 2.3]{chinburg-geodesics} implies that every hyperbolic $\gamma \in \Gamma$ has a non-real eigenvalue. Therefore, the existence of the needed hyperbolic $\gamma \in \Gamma$ follows directly from Proposition \ref{proposition:derivedcovolume}.

In light of the above it remains only to bound $d_L$ in terms of $V$ and put all of our estimates together. By \cite[Thm 4-A]{effectiveGW} (see also \cite{Wang2}) we may assume that the conductor $\ff_{L/k}$ of the extension $L/k$ satisfies
\begin{align*}
\abs{\ff_{L/k}} &\leq 64^{n_k} B(\Omega_k) \abs{\disc(B)}^{2n_k}\cdot \left[d_k^{n_k!-1}{n_k!}^{n_k!}\right]^{2A}\cdot \left[n_k!\log\abs{\disc(B)}\right]^4 \\
& \leq 64^{n_k} 2^{n_k^3}d_k^{n_k} \abs{\disc(B)}^{2n_k}\cdot \left[d_k^{n_k!-1}{n_k!}^{n_k!}\right]^{2A}\cdot \left[n_k!\log\abs{\disc(B)}\right]^4,
\end{align*}
where the latter inequality follows from Lemma \ref{lem:integralbasisbound}. The conductor-discriminant formula \cite[Ch VII, (11.9)]{N} and the fact that $d_L=\abs{\Delta_{L/k}}d_k^2$ implies that
\[ d_L\leq 64^{n_k} 2^{n_k^3}d_k^{n_k+2} \abs{\disc(B)}^{2n_k}\cdot \left[d_k^{n_k!-1}{n_k!}^{n_k!}\right]^{2A}\cdot \left[n_k!\log\abs{\disc(B)}\right]^4. \]
It now follows that
\[ \ell(\gamma)\leq e^{C_1V}\left[64^{n_k} 2^{n_k^3}d_k^{n_k+2} \abs{\disc(B)}^{2n_k}\cdot \left[d_k^{n_k!-1}{n_k!}^{n_k!}\right]^{2A}\cdot \left[n_k!\log\abs{\disc(B)}\right]^4\right]^{C_2+\log(V)}. \]
In order to bound this expression from above we will make use of the following three inequalities:
\begin{compactenum}
\item $n_k\leq 23+\log(V)$ (proven in \cite[Lemma 4.3]{chinburg-smallestorbifold}),
\item $d_k\leq V^{22}$ (proven as a part of \cite[Thm 4.1]{linowitz-isospectralsize}),
\item $\abs{\disc(B)}\leq 10^{57}V^7$ (proven in Lemma \ref{lem:Bdiscbound}).
\end{compactenum}
Substituting in these upper bounds, an elementary computation shows that $\ell_0(\gamma)\leq Ce^{\left(\log(V)^{\log(V)}\right)}$ for some absolute constant $C$. We use here that the term $n_k!^{n_k!}$ essentially dominates over all of the others and that its size can be estimated using Stirling's formula. The proposition follows.
\end{proof}

\subsection{Proof of Theorem \ref{theorem:recognizingalgebras}}

We begin with a lemma which is needed in the proof of Theorem \ref{theorem:recognizingalgebras}.

\begin{lem}\label{lem:thetabound}
For all $x>2$ we have $\prod_{p\leq x}p \leq e^{\frac{21x}{\log^3(x)}+x}$.	
\end{lem}

Setting $P(x)=\prod_{p\leq x}p$, $\log(P(x))$ is the Chebyshev theta function and the lemma follows from \cite[Thm 5.2]{dusart}. We now prove Theorem \ref{theorem:recognizingalgebras} from the introduction.

\begin{proof}[Proof of Theorem \ref{theorem:recognizingalgebras}]
If $B\not\cong B^\prime$, interchanging $B,B^\prime$ if necessary, we may assume that there exists a prime $\frakp$ of $k$ which ramifies in $B$ but not in $B^\prime$. By hypothesis if $\frakp$ is not real archimedean then $|\frakp|<x$. Let $L/k$ be a quadratic extension such that:
	\begin{compactenum}
	\item $[L_{\mathscr Q}:k_\frakq]=2$ for all primes $\frakq$ of $k$ with $|\frakq|<x$, $\frakq\neq \frakp$ and all primes $\mathscr Q$ of $L$ lying above $\frakq$;
	\item $[L_{\mathscr P}:k_\frakp]=1$ for all primes $\mathscr P$ of $L$ lying above $\frakp$; and
	\item all real places of $k$ not equal to $\frakp$ ramify in $L/k$.
	\end{compactenum}
The existence of such an extension $L/k$ follows from the Grunwald--Wang theorem. Using \cite[Ch 4]{effectiveGW} (see also \cite{Wang2}), we can find such an extension $L/k$ whose conductor $\ff_{L/k}$ satisfies
	\begin{equation}\label{equation:GWbound}
	\abs{\ff_{L/k}} \leq (32)^{n_k^2}B(\Omega)\left(\prod_{p\leq x}p\right)^{2n_k}.
	\end{equation}
Lemmas \ref{lem:integralbasisbound}, \ref{lem:thetabound}, and the conductor-discriminant formula imply that the relative discriminant $\Delta_{L/k}$ has norm less than the bound given in the theorem's statement. The proof of the theorem now follows from the Albert--Brauer--Hasse--Noether theorem, which implies that $B^\prime$ admits an embedding of $L/k$ whereas $B$ does not. \end{proof}

\subsection{Proof of Theorem \ref{thm:effectiveCHLR}}\label{mainproofsubsection}

In this subsection, we prove Theorem \ref{thm:effectiveCHLR}. We start with a proposition.

\begin{prop}\label{proposition:finalbounds}
Let $\Gamma$ be an arithmetic Fuchsian or Kleinian group with trace field $k$, quaternion algebra $B$, and covolume less than $V$. Let $L/k$ be a quadratic extension which embeds into $B$. We additionally suppose that $L$ is not totally complex in the case that $\Gamma$ is a Fuchsian group. If $\abs{\Delta_{L/k}}$ is less than the bound in the statement of Theorem \ref{theorem:recognizingalgebras} (applied with $X=10^{930}V^{130}$ if $\Gamma$ is a Fuchsian group and $X=10^{57}V^7$ if $\Gamma$ is a Kleinian group), then there exists a hyperbolic $\gamma\in\Gamma$ and absolute, effectively computable constants $c_1,c_2$ such that $L=k(\lambda(\gamma))$ and $\ell_0(\gamma)\leq c_1e^{c_2\log(V)V^{\alpha}}$, where $\alpha=130$ if $\Gamma$ is a Fuchsian group and is equal to $7$ otherwise.
\end{prop}

\begin{proof}
By Proposition \ref{proposition:derivedcovolume}, there exists a hyperbolic $\gamma\in\Gamma$ such that $L=k(\lambda_\gamma)$ and with length at most $e^{C_1V}d_L^{C_2+\log(V)}$ for absolute, effectively computable constants $C_1,C_2$. The result now follows from our hypothesis about $\abs{\Delta_{L/k}}$, the formula $d_L= \abs{\Delta_{L/k}}d_k^2$ and the fact that $d_k\leq V^{22}$.
\end{proof}

We are now ready to prove Theorem \ref{thm:effectiveCHLR}. In what follows, $c_1,c_2$ are the constants appearing in Proposition \ref{proposition:finalbounds}, $C$ is the constant appearing in Proposition \ref{proposition:derivednonrealeigenvalue} and $c\geq C$ is such that $c_1e^{c_2\log(V)V^7}\leq ce^{\left(\log(V)^{\log(V)}\right)}$ for all $V\geq 0.9$. Note that by Chinburg--Friedman--Jones--Reid \cite{chinburg-smallestmanifold}, every arithmetic hyperbolic $3$-manifold has volume $V>0.94$.

\begin{proof}[Proof of Theorem \ref{thm:effectiveCHLR}]
We prove Theorem \ref{thm:effectiveCHLR} in the case that the manifolds $M_i$ are $3$--manifolds and then make a few remarks regarding the (minor) modifications needed for the $2$--manifold case. Let $\Gamma_j = \pi_1(M_j)$ for $j=1,2$. As in Reid's proof that isospectral arithmetic $3$--manifolds are commensurable \cite{Reid-Isospectral}, it suffices to show that the quaternion algebras from which $\Gamma_1,\Gamma_2$ arise are isomorphic. To that end, let $(k_1,B_1)$ and $(k_2,B_2)$ be the number fields and quaternion algebras associated to $\Gamma_1,\Gamma_2$. By Proposition \ref{proposition:derivednonrealeigenvalue} there are hyperbolic $\gamma_1\in\Gamma_1, \gamma_2\in\Gamma_2$ with non-real eigenvalues $\lambda_{\gamma_1}$ and $\lambda_{\gamma_2}$ whose associated geodesics have the same complex length. Taking powers of $\gamma_1,\gamma_2$ if necessary, we may assume that $\gamma_1\in\Gamma_1^{(2)}, \gamma_2\in\Gamma_2^{(2)}$. By \eqref{TraceLengthFormula}, we have $\tr(\gamma_1) = \pm \tr(\gamma_2)$ and consequently that the images in $\C$ of $k_1,k_2$ coincide \cite[Lemma 2.3]{chinburg-geodesics}. Hence, $B_1,B_2/k$ are defined over a common number field $k$. To prove that $B_1,B_2$ are isomorphic, by Lemma \ref{lem:Bdiscbound} we have $\abs{\disc(B)}, \abs{\disc(B')}<10^{57}V^7$. Let $L/k$ be a quadratic extension which embeds into $B_1$ and with $\abs{\Delta_{L/k}}$ less than the bound given in Theorem \ref{theorem:recognizingalgebras}; we take $x=10^{57}V^7$. By Proposition \ref{proposition:finalbounds}, there exists $u_1\in B$ such that $L=k(u_1)$ with the property that the image $\gamma_1 \in \Gamma_1$ of $u_1$ in $\MM(2,\C)$ is a hyperbolic and satisfies
\[ \ell_0(\gamma_1)\leq c_1e^{c_2\log(V)V^7}\leq ce^{\left(\log(V)^{\log(V)}\right)}. \]
By hypothesis, there exists $\gamma_2\in\Gamma_2$ such that $\ell(\gamma_1)=\ell(\gamma_2)$. Let $u_2$ be a preimage of $\gamma_2$ in $B_2$. By \eqref{TraceLengthFormula}, we see that $\tr(\gamma_1)= \pm \tr(\gamma_2)$. Since the fields $k(u_1)$ and $k(u_2)$ are both isomorphic to $L$, we see that $B_2$ admits an embedding of $L$. The same argument shows that if $L^\prime/k$ is a quadratic extension which embeds into $B_2$ and has $\abs{\Delta_{L'/k}}$ less than the bound given in Theorem \ref{theorem:recognizingalgebras} with $x=10^{57}V^7$, then $B_1$ admits an embedding of $L^\prime$. Theorem \ref{theorem:recognizingalgebras} now shows that $B_1\cong B_2$, finishing our proof.
\end{proof}

\begin{rmk}
We briefly comment on the modifications needed for the $2$--dimensional case of Theorem \ref{thm:effectiveCHLR}.

As was noted in the $3$--dimensional case, it suffices to show that the quaternion algebras associated to $\Gamma_1,\Gamma_2$ are isomorphic. By Proposition \ref{proposition:finalbounds} there exists a hyperbolic $\gamma\in\Gamma_1$ such that $\ell(\gamma)\leq c_1e^{c_2\log(V)V^{130}}$. Proposition \ref{prop:fuchsiantracefield} shows that we further have $k_1=\Q(\tr(\gamma))$. By hypothesis there exists $\gamma^\prime\in\Gamma_2$ such that $\ell(\gamma)=\ell(\gamma^\prime)$, hence $\tr(\gamma)=\pm \tr(\gamma^\prime)$. Since $\Q(\tr(\gamma))=\Q(\tr(\gamma^\prime))$, we may assume, as above, that $k_1=k_2$. The remainder of the proof is analogous to the proof of the $3$--dimensional case; here, we can show that $B_1\cong B_2$ by proving that all maximal subfields $L$ of these algebras that are not totally complex and have $\abs{\Delta_{L/k}}$ less than the bound in Theorem \ref{theorem:recognizingalgebras} coincide.
\end{rmk}

We conclude this section by proving the following strengthening of Theorem \ref{thm:effectiveCHLR} in the case that the groups $\Gamma_i$ are derived from orders in quaternion algebras.

\begin{thm}\label{thm:arithmeticbusercourtois}
Let $k_1, k_2$ be totally real number fields (resp., number fields with a unique complex place), $B_1,B_2$ be quaternion division algebras over $k_1,k_2$ which are ramified at all but one real place (resp., all real places) of $k_1,k_2$, $\calO_1,\calO_2$ be maximal orders in $B_1,B_2$, and $V$ be such that $\covol(\Gamma_{\calO_1}),\covol(\Gamma_{\calO_2})\leq V$. There exist absolute effectively computable constants $c_1,c_2,c_3$ such that if the length sets (resp., complex length sets) of $\mathbf{H}^2/\Gamma_{\calO_1}, \mathbf{H}^2/\Gamma_{\calO_2}$ (resp., $\mathbf{H}^3/\Gamma_{\calO_1}, \mathbf{H}^3/\Gamma_{\calO_2}$) agrees for all lengths less than $c_1e^{c_2\log(V)V^{130}}$ (resp., $c_3e^{\left(\log(V)^{\log(V)}\right)}$) then $\mathbf{H}^2/\Gamma_{\calO_1}$ and $\mathbf{H}^2/\Gamma_{\calO_2}$ (resp., $\mathbf{H}^3/\Gamma_{\calO_1}$ and $\mathbf{H}^3/\Gamma_{\calO_2}$) are length-isospectral (resp., complex length-isospectral).
\end{thm}

\begin{proof}
Theorem \ref{thm:effectiveCHLR} and its proof show that $k_1\cong k_2, B_1\cong B_2$, hence $\calO_2$ is isomorphic to a maximal order $\calO\subset B_1$. If $\calO\cong\calO_1$, then $\mathbf{H}^2/\Gamma_{\calO_1},\textbf{H}^2/\Gamma_{\calO_2}$ (or $\mathbf{H}^3/\Gamma_{\calO_1}, \mathbf{H}^3/\Gamma_{\calO_2}$) will be isometric, hence isospectral. Suppose that $\mathcal O\not\cong \calO_1$ and $\mathbf{H}^2/\Gamma_{\mathcal O_1},\mathbf{H}^2/\Gamma_{\calO_2}$ (or  $\mathbf{H}^3/\Gamma_{\calO_1}$ and $\mathbf{H}^3/\Gamma_{\mathcal O_2}$) are not isospectral. By \cite[Thm 3.3]{Chinburg-Friedman-selectivity} and \cite[Thm 12.4.5]{MR}, there exists a quadratic extension $L/k$ which is unramified at all finite places (and which is not totally complex if the field $k_1$ is totally real) and a quadratic order $\Omega=\calO_k[\gamma]\subset L$ such that $\Omega$ embeds into exactly one of $\set{\calO_1,\calO}$. By Proposition \ref{proposition:derivedcovolume}, there exist absolute constants $C_1,C_2$ and a length (resp., complex length)
\begin{equation}\label{linoeq}
\ell_0(\gamma)\leq e^{C_1V}d_k^{2C_2+2\log(V)}\leq e^{C_1V} V^{44(C_2+2\log(V))},
\end{equation}
which lies in the length set (resp., complex length set) of exactly one of $\set{\mathbf{H}^2/\Gamma_{\calO},\mathbf{H}^2/\Gamma_{\calO_1}}$ (resp.,  $\set{\mathbf{H}^3/\Gamma_{\calO},\mathbf{H}^3/\Gamma_{\calO_1}}$). We note that the latter inequality (\ref{linoeq}) follows from the proof of \cite[Thm 4.1]{linowitz-isospectralsize}. By choosing constants appropriately, we contradict our hypothesis that the length set (resp., complex length sets) of $\mathbf{H}^2/\Gamma_{\calO_1}$ and $\mathbf{H}^2/\Gamma_{\calO_2}$ (or $\mathbf{H}^3/\Gamma_{\calO_1}$ and $\mathbf{H}^3/\Gamma_{\calO_2}$) coincide for all sufficiently small lengths.
\end{proof}

\section{Geometric submanifolds: Effective rigidity and asymptotic growth of surfaces}

We now turn our attention to an effective version of \cite[Thm 1.1]{McReid} which stated that two arithmetic hyperbolic 3--manifolds with the same totally geodesic surfaces are commensurable provided they have a totally geodesic surface.

\subsection{Proof of Theorem \ref{thm:recognizingalgebrasfromsubalgebras}}

To prove Theorem \ref{thm:recognizingalgebrasfromsubalgebras}, we require the following easy extension of \cite[Thm 9.55]{MR}.


\begin{thm}\label{thm:MR955extension}
Let $L$ be a number field and let $B/L$ be a quaternion algebra which is ramified precisely at the real places $\nu_1,\dots,\nu_s$ of $L$, let $k < L$ such that $[L:k]=2$, and $B_0/k$ be a quaternion algebra which is ramified at $\nu_1\mid_k,\dots,\nu_s\mid_k$ and at no other real places of $k$. Then $B\cong B_0\otimes_k L$ if and only if $\Ram_f(B)$ consists of the $2r$ distinct places $\{\mathfrak{P}_1,\mathfrak{P}_1^\prime,\dots,\mathfrak{P}_r,\mathfrak{P}_r^\prime\}$, where $\mathfrak{P}_i\cap \calO_k = \mathfrak{P}^\prime_i\cap \calO_k=\pp_i$ and $\Ram_f(B_0)\supset \{ \pp_1,\dots,\pp_r\}$ with $\Ram_f(B)\setminus\{\pp_1,\dots,\pp_r\}$ consisting of primes in $\calO_k$ which are either ramified or inert in the extension $L/k$.
\end{thm}

\begin{proof}[Proof of Theorem \ref{thm:recognizingalgebrasfromsubalgebras}]
Let $R_1$ denote the set of places of $L_1$ which ramify in $B_1$, $R_2$ denote the set of places of $L_2$ which ramify in $B_2$ and $R_i^\prime$ (for $i=1,2$) denote the set of places of $k$ lying below a place in $R_i$. As $B_0\otimes_k L_1\cong B_1$ and $B_0\otimes_k L_2\cong B_2$, it suffices to show that $L_1\cong L_2$. To that end, suppose that $L_1\not\cong L_2$ and let $L = L_1L_2$, which is Galois over $k$ with Galois group $(\Z/2\Z)\times (\Z/2\Z)$. Elementary properties of Frobenius elements \cite[Ch 10]{Lang-ANT} show that if $\pp_k$ is a prime of $k$ which is unramified in $L/k$ and whose Frobenius element $(\pp_k, L/k)$ corresponds to the element $(1,1)$ of $\Gal(L/k)$ then $\pp_k$ is inert in both $L_1/k$ and $L_2/k$. Similarly, if the Frobenius element $(\pp_k, L/k)$ corresponds to the element $(1,0)$ of $\Gal(L/k)$ then $\pp_k$ is inert in $L_1/k$ and splits in $L_2/k$. It now follows from the bound on the least prime ideal in the Chebotarev density theorem \cite{effectiveCDT} (see also Theorem \ref{thm:effectiveCDT}) that there exist primes $\omega_1,\omega_2$ of $k$ such that
\begin{compactenum}
\item $\omega_1$ is inert in $L_1/k$ and splits in $L_2/k$,
\item $\omega_2$ is inert in both $L_1/k$ and $L_2/k$,
\item Neither $\omega_1$ nor $\omega_2$ lie in $R_1^\prime \cup R_2^\prime$,
\item $\abs{\omega_1},\abs{\omega_2} \leq d_k^C (2\log(\abs{\disc(B_1)}\abs{\disc(B_2)}))^2$.
\end{compactenum}
Let $B'/k$ be a quaternion algebra such that $\Ram_\infty(B')=\{\nu\mid_k : \nu\in\Ram_\infty(B_1)\}$ and which ramifies at all primes lying in $R_1^\prime\cup \{\omega_1\}$ (and possibly at $\omega_2$ as well if needed for parity reasons). As $B_0\otimes_k L_1\cong B_1$, we deduce from Theorem \ref{thm:MR955extension} that $B'\otimes_k L_1\cong B_1$. Recall that $\omega_1$ splits in $L_2/k$. We can therefore write $\omega_1=\nu_1\nu_2$ for primes $\nu_1,\nu_2$ of $L_2$. Then $(L_2)_{\nu_1}\cong k_{\omega_1}$, which implies that $B_2\otimes_{L_2} (L_2)_{\nu_1}\cong B'\otimes_k k_{\omega_1}$, as by assumption we have $B'\otimes_k L_2\cong B_2$. As $\omega_1$ ramifies in $B'$ we deduce that $\nu_1$ ramifies in $B_2$, hence $\omega_1\in R_2^\prime$. This is a contradiction and so $L_1\cong L_2$.
\end{proof}

\subsection{Proof of Theorem \ref{EffectiveMcReid}}

We begin with lemma about the coarea of certain arithmetic Fuchsian groups.

\begin{lem}\label{lem:fuchsianvolbound}
Let $k$ be totally real, $B/k$ be a quaternion algebra, and $\calO \su B$ be a maximal order. Then $\coarea(\Gamma_\calO)\leq 2\pi^2 \abs{\disc(B)}$.
\end{lem}

\begin{proof}
Borel's volume formula \cite{borel-commensurability} (see also \cite[Ch 11]{MR}) shows that
\[ \coarea(\Gamma_\calO)=\frac{8\pi^2\zeta_k(2)\prod_{\pp\mid \disc(B)}(\abs{\pp}-1)}{(4\pi^2)^{n_k}}. \]
The lemma now follows from the well-known inequality $\zeta_k(s)\leq \zeta(s)^{n_k}$.
\end{proof}

\begin{prop}\label{proposition:ggsprop}
Let $M=\mathbf{H}^3/\Gamma$ be an arithmetic hyperbolic $3$--manifold with trace field $L$, quaternion algebra $B$, and volume $V$. Suppose that $k = L^+$ is the maximal totally real subfield of $L$ and that $[L:k]=2$. Let $B_0$ be a quaternion algebra over $k$ such that $B_0\otimes_k L\cong B$. Then there exists an absolute effectively computable constant $C$ such that $M$ contains a totally geodesic surface with area at most $2\pi^2\abs{\disc(B_0)}e^{CV}$.
\end{prop}

\begin{proof}
Let $\calO_0 \su B_0$ be a maximal order, $\calO \su B$ a maximal order such that $\Gamma_{\calO_0}\subset \Gamma_{\calO}$, and define $\Delta=\Gamma_{\calO_0}\cap \Gamma^{(2)}$. Then $\Delta$ is a Fuchsian group contained in $\Gamma$ and we have $[\Gamma_{\calO_0}:\Delta] \leq [\Gamma:\Gamma^{(2)}]$. By Lemma \ref{lem:gammaindexbound}, $[\Gamma:\Gamma^{(2)}]\leq e^{CV}$, where $C$ is an absolute effectively computable constant. The proposition now follows from Lemma \ref{lem:fuchsianvolbound}.
\end{proof}

For a hyperbolic $3$--manifold $M$, we denote by $GS(M)$ the collection of isometry types of totally geodesic surfaces. We can prove our main result of this section, an effective version of \cite[Thm 1.1]{McReid}.

\begin{proof}[Proof of Theorem \ref{EffectiveMcReid}]
Let $M_1=\mathbf{H}^3/\Gamma_1$, $M_2=\mathbf{H}^3/\Gamma_2$ and $B_1/L_1, B_2/L_2$ be the quaternion algebras and trace fields of $M_1, M_2$. Since $GS(M_1) \ne \emptyset$, $\Gamma_1$ contains a non-elementary Fuchsian group. By considering the quaternion algebra and trace field of this Fuchsian group we see that the maximal totally real subfield $k$ of $L_1$ satisfies $[L_1:k]=2$. As $GS(M_1)\cap GS(M_2)$ is non-empty, we see that $[L_2:k]=2$ and also that there exists a quaternion algebra $B_0$ over $k$ such that $B_0\otimes_k L_1\cong B_1$ and $B_0\otimes_k L_2\cong B_2$. Let $C$ be the constant appearing in the bound on the least prime ideal in the Chebotarev density theorem \cite{effectiveCDT} (see also Theorem \ref{thm:effectiveCDT}). By combining the estimates $d_k\leq d_{L_1}\leq V^{22}$ and $\abs{\disc(B_1)},\abs{\disc(B_2)}\leq 10^{57}V^7$ with an elementary computation, there exists an absolute constant $C_1$ such that
\[ V^{C_1}\geq d_k^{2C} (2\log(\abs{\disc(B_1)}\abs{\disc(B_2)}))^4\abs{\disc(B_1)}\abs{\disc(B_2)}. \]
Let $C_2$ be the constant appearing in Proposition \ref{proposition:ggsprop} and choose $C_3$ so that $2\pi^2 V^{C_1}e^{C_2V}\leq e^{C_3V}$;  note that $C_3$ may be chosen independently of $M_1,M_2$ and $V$. We will now show that if a finite type hyperbolic surface $X$ lies in $GS(M_1)$ if and only it lies in $GS(M_2)$ whenever the area of $X$ is less than $e^{C_3V}$, then $M_1$ and $M_2$ are commensurable. Let $B'$ be a quaternion algebra over $k$ which is ramified at all real places of $k$ except the identity and satisfies
\[ \abs{\disc(B')}\leq d_k^{2C} (2\log(\abs{\disc(B_1)}\abs{\disc(B_2)}))^4\abs{\disc(B_1)}\abs{\disc(B_2)} \]
as well as $B'\otimes_k L_1\cong B_1$. Proposition \ref{proposition:ggsprop} and the discussion above show that $M_1$ contains a totally geodesic surface $X$ (arising from the quaternion algebra $B$) with area at most $e^{C_3V}$. Our assumption implies that $M_2$ contains a totally geodesic surface isometric to $X$ as well. Consequently we must have $B'\otimes_k L_2\cong B_2$. Interchanging the roles of $B_1,B_2$, we see that by Theorem \ref{thm:recognizingalgebrasfromsubalgebras}, $L_1\cong L_2$ and $B_1\cong B_2$. Hence by Theorem \ref{CommensurabilityInvariants}, $M_1,M_2$ are commensurable.
\end{proof}

\subsection{Proof of Theorem \ref{thm:infinitelymanytotallygeodesics}}

We conclude this section with a proof of Theorem \ref{thm:infinitelymanytotallygeodesics}.

\begin{proof}[Proof of Theorem \ref{thm:infinitelymanytotallygeodesics}]
Setting $k = L^+$, as $M$ contains a totally geodesic surface, we have $[L:k]=2$. Theorem \ref{TGS-Theorem}, along with a slight modification to the proof of Theorem \ref{thm:quatembedding} in the case that $r=1$, shows that there exists a constant $c(L)$ depending only on $L$ (the constant only depends on $k$) such that for sufficiently large $x$ the number of quaternion algebras $B'/k$ satisfying $B'\otimes_k L\cong B$ and $\abs{\disc(B')}\leq x$ is asymptotic to $\left[c(L)\disc(B)^{1/2}\right]x/\log(x)^{1/2}$. The theorem now follows from Proposition  \ref{proposition:ggsprop}.
\end{proof}

\end{document}